\documentclass[12pt]{amsart}

\usepackage[hmargin=2.8cm,bmargin=2.8cm,tmargin=3.5cm]{geometry}

\usepackage{amsmath,amssymb,amsthm,amsfonts,enumerate,url}
\usepackage{mathrsfs}

\usepackage{hyperref}

\usepackage{tikz}
\usepackage{float}
\usetikzlibrary{automata,positioning,arrows.meta}
\theoremstyle{plain}
\newtheorem{theorem}{Theorem}[section]
\newtheorem{lemma}[theorem]{Lemma}
\newtheorem{proposition}[theorem]{Proposition}

\theoremstyle{definition}
\newtheorem{definition}[theorem]{Definition}

\newtheorem{conjecture}[theorem]{Conjecture}

\newtheorem{remark}[theorem]{Remark}

\numberwithin{equation}{section}
\newcommand{\R}{\mathbb{R}}

\DeclareMathOperator{\diam}{diam}

\DeclareMathOperator{\dist}{dist}

\DeclareMathOperator{\aufspan}{span}

\newcommand{\deig}{\lambda}
\newcommand{\neig}{\mu}

\title[A family of eigenvalue bounds for quantum graphs]{A family of diameter-based eigenvalue bounds for quantum graphs}% in terms of diameter} 



\author{J.~B.~Kennedy}

\address{Grupo de F\'isica Matem\'atica, Faculdade de Ci\^encias, Universidade de Lisboa, Campo Grande, Edif\'{i}cio~C6, \mbox{P-1749-016} Lisboa, Portugal}
\email{jbkennedy@fc.ul.pt}

\thanks{The work of the author was supported by the Funda{\c{c}}{\~a}o para a Ci{\^e}ncia e a Tecnologia, Portugal, via the program ``Investigador FCT'', reference IF/01461/2015, and project PTDC/MAT-CAL/4334/2014, as well as by the Center for Interdisciplinary Research (ZiF) in Bielefeld, Germany, within the framework of the cooperation group on ``Discrete and continuous models in the theory of networks''. The author would like to extend his warmest thanks to Gregory Berkolaiko, Pavel Kurasov and Delio Mugnolo for innumerable discussions and much inspiration over the last few years, including in the context of the current work. He wishes to thank Delio Mugnolo in particular for pointing out an analogous result from the discrete literature (Proposition~\ref{prop:mohar}), Pavel Kurasov for helpful advice in the context of Proposition~\ref{prop:star}, and the anonymous referee for a number of constructive suggestions, including the encouragement to consider the case of the higher eigenvalues.
}

\begin{document}

\begin{abstract}
We establish a sharp lower bound on the first non-trivial eigenvalue of the Laplacian on a metric graph equipped with natural (i.e., continuity and Kirchhoff) vertex conditions in terms of the diameter and the total length of the graph. This extends a result of, and resolves an open problem from, [J.~B.~Kennedy, P.~Kurasov, G.~Malenov\'a and D.~Mugnolo, Ann.~Henri Poincar\'e \textbf{17} (2016), 2439--2473, Section~7.2], and also complements an analogous lower bound for the corresponding eigenvalue of the combinatorial Laplacian on a discrete graph. We also give a family of corresponding lower bounds for the higher eigenvalues under the assumption that the total length of the graph is sufficiently large compared with its diameter. These inequalities are sharp in the case of trees.
\end{abstract}

\maketitle

\section{Introduction}
\label{sec:intro}

Let $\mathcal{G}$ be a connected, compact metric graph with a finite number of edges and let $-\Delta$ denote the Laplacian operator on $L^2(\mathcal{G})$ with natural (i.e., continuity and Kirchhoff) vertex conditions.\footnote{We recall that these conditions are also called standard, Neumann--Kirchhoff or even just Neumann conditions in the literature.} Since $-\Delta$ can be shown by standard means to be a self-adjoint operator with compact resolvent, one obtains the existence of a discrete sequence of eigenvalues of this operator, which we think of as eigenvalues of the \emph{quantum graph} itself, having the form
\begin{equation}
\label{eq:eigenvalue-sequence}
	0 = \neig_1 (\mathcal{G}) < \neig_2 (\mathcal{G}) \leq \neig_3 (\mathcal{G}) \leq \ldots \to \infty;
\end{equation}
the corresponding eigenfunctions may be chosen to form an orthonormal basis of $L^2(\mathcal{G})$. We refer to the monographs \cite{bk13,m14} and the seminal review article \cite{gs06} as well as Section~\ref{sec:background} for more details.

It is a major preoccupation of spectral geometry to investigate how the sequence of eigenvalues \eqref{eq:eigenvalue-sequence} of a differential operator such as the Laplacian depends on the structure, be it total size, shape, degree of connectivity etc., of the underlying object on which it is defined. For operators on domains and manifolds, this goes back at least to conjectures of Saint Venant and Lord Rayleigh in the mid-to-late 19th Century (see \cite{p67}; we refer also to \cite{h17,h03} for more modern overviews of the field). In the case of quantum graphs, that is, metric graphs with a differential operator defined on them, the first work in this direction appeared about 30 years ago \cite{n87}, where it was proved that the first non-trivial eigenvalue $\neig_2 (\mathcal{G})$ of the Laplacian with natural conditions on a graph whose total length, i.e., the sum of all its edge lengths, is $L>0$ satisfies
\begin{equation}
\label{eq:nicaise}
	\neig_2 (\mathcal{G}) \geq \frac{\pi^2}{L^2},
\end{equation}
the right-hand side corresponding to the first non-trivial eigenvalue on an interval of the same total length $L$ as $\mathcal{G}$. After a lull in the 1990s and early 2000s, in the last few years there seems to have been an explosion of interest in the topic, as witnessed by the long list of works establishing bounds on some or all of the eigenvalues \eqref{eq:eigenvalue-sequence}, for example in terms of the total length, diameter, number of edges or vertices, edge connectivity,\ldots \ of the graph, or establishing properties of extremising graphs realising the bounds, or developing tools with which the eigenvalues can be manipulated, or else considering similar problems for related nonlinear operators. We refer to {\cite{ast17,ast15a,ast15b,assw17,a16,bl17,bkkm18,bkkm17,d18,f05,kkmm16,km16,kn14,kmn13,pr16,r17,rs18}} and mention in particular the generalisation of \eqref{eq:nicaise} to the higher eigenvalues \cite{f05}
\begin{equation}
\label{eq:friedlander}
	\neig_k (\mathcal{G}) \geq \frac{\pi^2 (k-1)^2}{L^2},
\end{equation}
with equality if and only if $\mathcal{G}$ is an \emph{equilateral $k$-star}, a graph consisting of $k$ edges of equal length $L/k$, all joined together at exactly one common vertex.

The goal of the present contribution is to give lower a bound on $\neig_k (\mathcal{G})$ in terms of the total length $L\in (0,\infty)$ of the graph $\mathcal{G}$ and its diameter
\begin{displaymath}
	D := \diam (\mathcal{G}) := \sup \{ \dist (x,y): x,y \in \mathcal{G} \} \in (0, L],
\end{displaymath}
where the distance is with respect to the canonical (Euclidean) metric in $\mathcal{G}$, i.e., the shortest Euclidean path within $\mathcal{G}$ connecting the points $x$ and $y$, and the supremum is in fact a maximum since $\mathcal{G}$ is assumed to be compact.

For $k=2$, this problem was first studied in \cite[Section~7.2]{kkmm16}, where a non-trivial but non-sharp lower bound for $\neig_2$ was given, and the question of obtaining the best possible bound was left open (see Remark~7.3(a) there). Here, by using some more advanced tools developed recently in \cite{bkkm18} (which we call \emph{surgery principles}), we can give a complete answer: our main theorem is as follows.

\begin{theorem}
\label{thm:dumbbell-diameter}
Assume that $\mathcal{G}$ is a connected, compact metric graph with a finite number of edges of total length $L>0$ and diameter $\diam (\mathcal{G}) = D \in (0,L)$. Then $\neig_2 (\mathcal{G})$ is larger than the square $\omega^2$ of the smallest positive solution $\omega > 0$ of the transcendental equation
\begin{equation}
\label{eq:dumbbell-limit}
	\cos \left(\frac{\omega D}{2}\right) = \frac{\omega (L-D)}{2} \sin \left(\frac{\omega D}{2}\right);
\end{equation}
the number $\omega^2$ satisfies the two-sided bound
\begin{equation}
\label{eq:neig2-bound}
	\frac{1}{LD} < \omega^2 < \frac{12}{LD}.
\end{equation}
Equality is never attained on any fixed graph, but there is a sequence of graphs $\mathcal{D}_n$ each of length $L$ and diameter $D$ such that $\neig_2 (\mathcal{D}_n) \to \omega^2$ as $n \to \infty$.
\end{theorem}

To describe our result for the higher eigenvalues $\mu_k$, $k\geq 3$, we first recall that the \emph{(first) Betti number} $\beta = \beta (\mathcal{G})$ of the graph $\mathcal{G}$ is defined to be the number of independent cycles of $\mathcal{G}$; equivalently, if $\mathcal{G}$ has $E$ edges and $V$ vertices, then it is given by $\beta = E - V +1$. In particular, we have $\beta = 0$ if and only if $\mathcal{G}$ is a tree. We will require that $L$ be ``large'' compared with $D$ in the sense that the quantity
\begin{equation}
\label{eq:parameter-constraint}
	\gamma (L,D,k,\beta) := \begin{cases} \frac{L}{k-\beta} - \frac{D}{2} \qquad &\text{if } k > \beta,\\
	\frac{L}{k} - \frac{D}{2} \qquad &\text{otherwise},\end{cases}
\end{equation}
will be assumed to be positive.

\begin{theorem}
\label{thm:diameter-higher}
Assume that $\mathcal{G}$ is a connected, compact metric graph with a finite number of edges of total length $L>0$ and diameter $\diam (\mathcal{G}) = D \in (0,L)$, and that the quantity $\gamma = \gamma (L,D,k,\beta)$ defined by \eqref{eq:parameter-constraint} is strictly positive. Further assume that no loop\footnote{By a \emph{loop} we mean an edge which starts and ends at the same vertex, possibly after the suppression of vertices of degree two. A precise definition is given in \cite[Definition~3.1]{beli17}.} in $\mathcal{G}$ is longer than $D$. Then $\neig_k (\mathcal{G})$ is larger than the square $\omega_{k,\beta}^2$ of the smallest positive solution $\omega = \omega_{k,\beta} > 0$ of the transcendental equation
\begin{equation}
\label{eq:higher-star-limit}
	\cos \left(\frac{\omega D}{2}\right) = \omega\gamma \sin \left(\frac{\omega D}{2}\right);
\end{equation}
the number $\omega_{k,\beta}^2$ satisfies the two-sided bound
\begin{equation}
\label{eq:neigk-bound}
	\frac{2}{D\gamma + \frac{D^2}{2}} \leq \omega^2 \leq \frac{2}{D\gamma - \frac{D^2}{6}}.
\end{equation}
There is a sequence of tree graphs $\mathcal{T}_n$ each of length $L$ and diameter $D$ such that $\neig_k (\mathcal{T}_n) \to \omega_{k,0}^2$ as $n \to \infty$.
\end{theorem}

While Theorem~\ref{thm:diameter-higher} is essentially optimal for trees, we expect that improvements may be possible if $\beta \geq 1$. We give some remarks to this effect in Section~\ref{sec:remarks}.

To describe concrete sequences $\mathcal{D}_n$ and $\mathcal{T}_n$ of optimisers, and at the same time to explain the meaning of \eqref{eq:dumbbell-limit} and \eqref{eq:higher-star-limit}, we first need to introduce a particular class of graphs, which will also play a role in the proofs.

\begin{definition}
\label{def:star}
Fix suitable numbers $n \in \mathbb{N}$, $n\geq 2$, and $0<D\leq L$. We denote by
\begin{displaymath}
	\mathcal{S}_n = \mathcal{S} (L,D,n)
\end{displaymath}
the unique star graph having $n+1$ edges, total length $L$ and total diameter $D$, such that there is one distinguished edge $e_0$ of length $\ell_0 = (nD-L)/(n-1)$ and $n$ identical edges of length $\ell_1 = (L-D)/(n-1)$ each, all joined at a common vertex (see Figure~\ref{fig:star}).
\begin{figure}[H]
\begin{tikzpicture}[scale=0.8]
\coordinate (a) at (0,0);
\coordinate (b) at (4,0);
\draw[fill] (b) circle (1.5pt);
\draw[thick] (a) -- (b);
\draw[thick] (b) -- (5,0);
\draw[thick] (b) -- (4,1);
\draw[thick] (b) -- (4,-1);
\draw[thick] (b) -- (4.71,0.71);
\draw[thick] (b) -- (4.71,-0.71);
\draw[thick] (b) -- (3.29,0.71);
\draw[thick] (b) -- (3.29,-0.71);
\draw[fill] (5,0) circle (1.5pt);
\draw[fill] (4,1) circle (1.5pt);
\draw[fill] (4,-1) circle (1.5pt);
\draw[fill] (4.71,0.71) circle (1.5pt);
\draw[fill] (4.71,-0.71) circle (1.5pt);
\draw[fill] (3.29,0.71) circle (1.5pt);
\draw[fill] (3.29,-0.71) circle (1.5pt);
\coordinate (c) at (9,0);
\coordinate (d) at (13,0);
\draw[fill] (d) circle (1.5pt);
\draw[thick] (c) -- (d);
\draw[thick] (d) -- (13.4,0);
\draw[fill] (13.4,0) circle (1.5pt);
\draw[thick] (d) -- (13,0.4);
\draw[fill] (13,0.4) circle (1.5pt);
\draw[thick] (d) -- (13,-0.4);
\draw[fill] (13,-0.4) circle (1.5pt);
\draw[thick] (d) -- (13.2,0.346);
\draw[fill] (13.2,0.346) circle (1.5pt);
\draw[thick] (d) -- (13.2,-0.346);
\draw[fill] (13.2,-0.346) circle (1.5pt);
\draw[thick] (d) -- (12.8,0.346);
\draw[fill] (12.8,0.346) circle (1.5pt);
\draw[thick] (d) -- (12.8,-0.346);
\draw[fill] (12.8,-0.346) circle (1.5pt);
\draw[thick] (d) -- (13.346,0.2);
\draw[fill] (13.346,0.2) circle (1.5pt);
\draw[thick] (d) -- (13.346,-0.2);
\draw[fill] (13.346,-0.2) circle (1.5pt);
\draw[thick] (d) -- (12.654,0.2);
\draw[fill] (12.654,0.2) circle (1.5pt);
\draw[thick] (d) -- (12.654,-0.2);
\draw[fill] (12.654,-0.2) circle (1.5pt);
\draw[thick,fill=white] (a) circle (2.5pt);
\draw[thick,fill=white] (c) circle (2.5pt);
\node at (1.75,0) [anchor=south] {$e_0$};
\node at (10.875,0) [anchor=south] {$e_0$};
\node at (0,0) [anchor=east] {$v_0$};
\node at (9,0) [anchor=east] {$v_0$};
\end{tikzpicture}
\caption{The stars $\mathcal{S}_7$ (left) and $\mathcal{S}_{11}$ (right), for given $L$ and $D$. The white circles at $v_0$ indicate Dirichlet vertex conditions.}
\label{fig:star}
\end{figure}
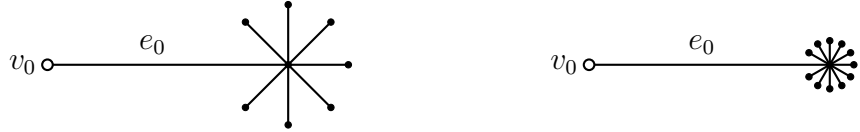
\end{definition}

Of interest will be the smallest eigenvalue
\begin{displaymath}
	\deig_1 (\mathcal{S}_n)
\end{displaymath}
of the Laplacian with a Dirichlet (zero) condition at the degree-one vertex $v_0$ at the end of $e_0$ and natural conditions at all other vertices (see Section~\ref{sec:background} for more details on our notation). Observe that, for fixed $L$ and $D$, as $n\to\infty$ the length $\ell_0$ of the edge $e_0$ to $D$, and the other edges contract to a point: in the limit, we have an interval of length $D$ with a kind of point mass of size $L-D$ at one endpoint. Henceforth, whenever we speak of stars, we shall \emph{always mean stars of this form}.

The link to Theorem~\ref{thm:dumbbell-diameter} is that the graphs $\mathcal{S}_n$ with length $L/2$ and diameter $D/2$ are the ``building blocks'' for a sequence of limiting domains $\mathcal{D}_n$. More precisely, we can form $\mathcal{D}_n$ by gluing together two copies of $\mathcal{S}_n$ at their respective Dirichlet vertices (the corresponding domain, a \emph{symmetric star dumbbell} in the language of \cite[Section~7.2]{kkmm16}, is pictured in Figure~\ref{fig:star-dumbbell} in Section~\ref{sec:stars}); we then have $\neig_2 (\mathcal{D}_n) = \deig_1 (\mathcal{S}_n)$ and each copy of $\mathcal{S}_n$ corresponds to a nodal domain of the eigenfunction of $\neig_2 (\mathcal{D}_n)$ (see Section~\ref{sec:stars} for more details). Similarly, the domains $\mathcal{T}_n$ can be formed by taking $k$ copies of $\mathcal{S}_n$ (each now with length $L/k$ and diameter $D/2$) and joining them at their common Dirichlet vertex; then $\neig_k (\mathcal{T}_n) = \deig_1 (\mathcal{S}_n)$. As $n\to\infty$, the $\mathcal{T}_n$ converge to an equilateral $k$-star with diameter $D$ and a point mass of size $\gamma$ at each vertex of degree one, thus recalling the equilateral $k$-stars which were the optimisers in the inequality \eqref{eq:friedlander}.

The next proposition summarises the properties of these stars $\mathcal{S}_n$, and in particular provides a rigorous justification of the formulae \eqref{eq:dumbbell-limit} and \eqref{eq:higher-star-limit}; it also implies the bounds on $\omega^2$ given in \eqref{eq:neig2-bound} in Theorem~\ref{thm:dumbbell-diameter}, and on $\omega_{k,\beta}^2$ in \eqref{eq:neigk-bound} in Theorem~\ref{thm:diameter-higher} (where $\mathcal{S}_n$ is chosen to have diameter $D/2$ and length $L/(k-\beta)$ if $k>\beta$ or $L/k$ otherwise).

\begin{proposition}
\label{prop:star}
Suppose $L>0$ and $D \in (0,L]$ are given and, for $n\geq 2$, $\mathcal{S}_n$ is the star graph described in Definition~\ref{def:star} having length $L$ and diameter $D$. Denote by $\deig_1 (\mathcal{S}_n)$ the first eigenvalue of the Laplacian on $\mathcal{S}_n$ with a Dirichlet condition at the degree one vertex $v_0$ of the edge $e_0$ and natural conditions at all other vertices. Then
\begin{enumerate}
\item the sequence $(\deig_1 (\mathcal{S}_n))_{n\geq 1}$ is strictly decreasing in $n$, and as $n\to\infty$, the eigenvalue $\deig_1 (\mathcal{S}_n)$ converges from above to the square $\omega^2$ of the smallest positive solution $\omega>0$ of the transcendental equation
\begin{equation}
\label{eq:star-limit}
	\cos (\omega D) = \omega (L-D)\sin(\omega D);
\end{equation}
\item the number $\omega^2$ from (1) is the smallest (strictly) positive eigenvalue of the problem
\begin{equation}
\label{eq:wentzell-star}
\begin{aligned}
	-u''(x) &= \omega^2 u(x) \qquad \text{in } (0,D),\\
	u(0) &=0, \\
	u''(D) + \frac{2}{L-D} u'(D) &=0;
\end{aligned}
\end{equation}
\item for fixed $L$, the number $\omega^2$ from (1) is a strictly decreasing function of $D \in (0,L]$;
\item the number $\omega^2$ from (1) satisfies the bounds
\begin{displaymath}
	\frac{4}{LD} < \frac{4}{LD-\frac{D^2}{2}} \leq \omega^2 \leq \frac{48}{3LD - 2D^2} < \frac{48}{LD}.
\end{displaymath}
\end{enumerate}
\end{proposition}

The condition in \eqref{eq:wentzell-star}, which is usually called a \emph{generalised Wentzell}-type boundary condition, reflects the concentration of mass at one endpoint of the star $\mathcal{S}_n$ as $n\to\infty$. (The term \emph{Wentzell boundary condition} is usually used to describe the situation where the differential operator, in this case the Laplacian, itself appears in the boundary condition. We refer to \cite{ampr03,mr07} for more information in the case of the Laplacian on domains.)

\begin{remark}
(a) The \emph{stars} $\mathcal{S}_n$ are not the only building blocks we could use to construct suitable $\mathcal{D}_n$ and $\mathcal{T}_n$: in principle, we simply need a sequence of domains converging in an appropriate sense to an interval of given diameter, with a suitable point mass at one end described by the Wentzell condition in \eqref{eq:wentzell-star}. These could, for example, be suitably chosen \emph{stowers} (see \cite[Example~1.5]{bl17}) with one long edge $e_0$ and short loops in place of short pendant edges. We will use stars as they are easier to handle in our context.

(b) In \cite[Section~7.2]{kkmm16}, in place of \eqref{eq:dumbbell-limit} the lower bound is the square $\tilde\omega^2$ of the smallest positive solution of
\begin{displaymath}
	\cos (2\tilde\omega D) = (L-2D)\tilde\omega \sin (2\tilde\omega D)
\end{displaymath}
as long as $D \leq L/2$; this number satisfies $\tilde\omega^2 > 1/(2LD)$.\footnote{There was an arithmetic error in the upper bound in \cite[Remark~7.3(a)]{kkmm16}; namely, it was too small by a factor of $4$.} There, proofs of statements corresponding to Proposition~\ref{prop:star}(1) and (2) are given (in a slightly different form). The derivation of the equation \eqref{eq:wentzell-star} from \eqref{eq:star-limit} is also described there; see \cite[Remark~7.3(c)]{kkmm16}. However, the proof of Theorem~\ref{thm:dumbbell-diameter} uses an essentially different set of tools from the proof of the corresponding main result \cite[Theorem~7.2]{kkmm16}. Indeed, here we will make use of both a new transplantation principle and an Hadamard-type length perturbation formula (see Section~\ref{sec:background} for details).
\end{remark}

Finally, we remark that Theorem~\ref{thm:dumbbell-diameter} and Proposition~\ref{prop:star} recall very much results for the first non-trivial eigenvalue of \emph{discrete} graph Laplacians (this eigenvalue is often called the \emph{algebraic connectivity} of the discrete graph), in terms of the number of vertices of the graph -- the discrete equivalent of its size, i.e., length -- as well as the (now integer-valued) diameter. We reproduce the statements here in our language for ease of comparison.

\begin{proposition}[\cite{fk98}, Corollary~3.3]
\label{prop:fallat-kirkland}
Let $\mathsf{T}$ be any (discrete) tree with diameter $D \geq 3$ and $V\geq 4$ vertices. Assume that $V-D$ is odd. Then the smallest non-trivial discrete Laplacian eigenvalue of $\mathsf{T}$ is at least as large as that of a symmetric star dumbbell formed by a chain of edges (``handle'') of length $D-2$, with $(V-D+1)/2$ pendant edges attached to each end. (If $V-D$ is even, one edge must be removed from one of the pendant stars.)
\end{proposition}

In fact, it seems plausible to expect that this result should hold for all graphs on $V$ vertices, not only trees (just as our result holds independently of the topology of the graph). To the best of our knowledge this has not been proved; however, there is a slightly older result which is valid for all graphs, which very much recalls our estimate \eqref{eq:neig2-bound}, and which is at least asymptotically optimal as $V \to \infty$.

\begin{proposition}[\cite{m91}, Theorem~4.2 and example after it]
\label{prop:mohar}
Let $\mathsf{G}$ be any discrete graph with diameter $D \geq 1$ and $V\geq 2$ vertices. Then its smallest non-trivial Laplacian eigenvalue is at least as large as $4/(DV)$. 
%\begin{equation}
%\label{eq:mohar}
%	\frac{4}{DV}.
%\end{equation}
Equality is achieved in the limit as $V\to \infty$ for fixed $D$ by discrete analogues of the symmetric star dumbbells described in Proposition~\ref{prop:fallat-kirkland}.
\end{proposition}

In fact, this bound was extended very recently to infinite discrete graphs equipped with a probability measure in place of the usual one (so that the total ``length'' is one) and finite diameter; see \cite[Corollary~3.7]{lss18}. It would be interesting to know whether the latter result could be extended to quantum graphs, and to know what happens in the case of the higher eigenvalues. We thank Delio Mugnolo for bringing these results to our attention.

In Section~\ref{sec:background} we will recall from \cite{bkkm18} the elementary but powerful technical tools we will need for the proofs; for the sake of readability we will provide proof sketches here but refer to \cite{bkkm18} for full details. In Section~\ref{sec:stars} we give the proof of Proposition~\ref{prop:star} together with a detailed analysis of the stars $\mathcal{S}_n$ as well as their counterparts, the symmetric star dumbbells $\mathcal{D}_n$, and how their eigenvalues depend on parameters like length and diameter. These will be needed in the proofs of Theorem~\ref{thm:dumbbell-diameter} and~\ref{thm:diameter-higher}, which are in Section~\ref{sec:proof}. Finally, in Section~\ref{sec:remarks}, we discuss the role of some of the assumptions in Theorem~\ref{thm:diameter-higher}, and the possibility that they may be weakened.

\section{Background results: surgical tools}
\label{sec:background}

In this section we recall both the formal definition of the Laplacian on a quantum graph and the characterisation of its eigenvalues, as well as the ``surgery'' tools we shall need from \cite{bkkm18}.

Formally, the metric graph $\mathcal{G}$ is taken to consist of a set of edges $\mathcal{E} = \{e_1,\ldots,e_{M}\}$, each of which may be identified with an interval $e_j \sim [0,\ell_j]$, $j=1,\ldots,M$, and a set of vertices $\mathcal{V} = \{v_1,\ldots,v_{N}\}$; we write $e_i \sim e_j$ if $e_i$ and $e_j$ are adjacent (share a vertex), and in a slight abuse of notation $e \sim v$ if the vertex $v$ is incident with the edge $e$, and $e \sim vw$ if $e$ runs from $v$ to $w$ (i.e., both are incident with $e$). We always assume our graph to be connected, but \emph{we explicitly allow it to have loops} ($e \sim vv$ for some $v \in \mathcal{V}$) and multiple edges running between two given vertices; in the latter case we speak of \emph{parallel} edges. A \emph{pendant edge} is any edge which ends at a vertex of degree one; the latter may be referred to as a \emph{pendant vertex}.

We consider the operator associated with the bilinear form $a: H^1 (\mathcal{G}) \times H^1 (\mathcal{G}) \to \R$,
\begin{equation}
\label{eq:form}
	a(f,g) := \int_\mathcal{G} f'g'\,\textrm{d}x \equiv \sum_{e \in \mathcal{E}} \int_e f'g'\,\textrm{d}x,
\end{equation}
where
\begin{displaymath}
	L^2 (\mathcal{G}) \simeq \bigoplus_{e \in \mathcal{E}} L^2(e),\qquad H^1 (\mathcal{G}) = \{f \in L^2(\mathcal{G}): f' \in L^2(\mathcal{G}) \}.
\end{displaymath}
Here $f'$ is to be interpreted in the distributional sense, and the space $H^1 (\mathcal{G}) \hookrightarrow C(\mathcal{G})$ in particular encodes both the vertex incidence relations and the vertex conditions. Indeed, the corresponding operator is given by the negative Laplacian (negative of the second derivative) on each edge. Its operator domain consists of those $H^1$-functions which, in addition to being automatically continuous across the vertices as members of $H^1(\mathcal{G})$, also satisfy the Kirchhoff condition
\begin{displaymath}
	\sum_{e \sim v} f|_e'(v) = 0
\end{displaymath}
at every vertex $v \in \mathcal{V}$, where $f|_e'$ is the derivative of the function along the edge $e$ pointing into $v$. The associated smallest non-trivial eigenvalue $\neig_2 (\mathcal{G})$, often also called the \emph{spectral gap} since $\neig_1 (\mathcal{G})=0$, admits the variational characterisation\footnote{Note that we use the numbering convention from \cite{bkkm18}. Both the notation and the numbering condition in \cite{kkmm16} are different; there, $\deig_1 (\mathcal{G}) > 0$ is the smallest non-trivial eigenvalue of the Laplacian with natural vertex conditions.}
\begin{equation}
\label{eq:lambda1}
	\neig_2 (\mathcal{G}) = \inf \left\{ \frac{\int_\mathcal{G} |f'|^2\,\textrm{d}x}{\int_\mathcal{G} |f|^2\,\textrm{d}x}:
	0 \neq f\in H^1(\mathcal{G}),\,\int_\mathcal{G} f\,\textrm{d}x=0\right\}
\end{equation}
with equality if and only if $f$ is a corresponding eigenfunction, which we will tend to denote by $\psi$. For the higher eigenvalues, the usual min-max characterisation of Courant--Fischer type is available: we have
\begin{equation}
\label{eq:higher-ev}
	\neig_k (\mathcal{G}) = \inf_{M \subset H^1 (\mathcal{G})}\, \max_{0 \neq f \in M}\, 
	\frac{\int_\mathcal{G} |f'|^2\,\textrm{d}x}{\int_\mathcal{G} |f|^2\,\textrm{d}x},
\end{equation}
where the infimum is taken over all subspaces of $H^1 (\mathcal{G})$ of dimension $k$, and equality is achieved by any set $M$ consisting of $k$ linearly independent eigenfunctions corresponding to $\neig_1,\ldots,\neig_k$ (see \cite[Sections~2 and~4.1]{bkkm18}, also for a characterisation of the sets achieving equality in \eqref{eq:higher-ev}, which is more complicated than for $\neig_2$ and seems little known).

If instead we wish to consider the Laplacian with Dirichlet (zero) vertices on a subset $\mathcal{V}_{\mathcal{D}} \subset \mathcal{V}$, our form is still given by \eqref{eq:form} but our form domain changes to $H^1_0 (\mathcal{G}) := \{ f \in H^1 (\mathcal{G}) : f(v)=0 \text{ for all } v \in \mathcal{V}_{\mathcal{D}} \}$ (the set $\mathcal{V}_{\mathcal{D}}$ being clear from the context). In this case, we will denote the eigenvalues by $0 < \deig_1 (\mathcal{G}) < \deig_2 (\mathcal{G}) \leq \ldots$, where the smallest eigenvalue is given by
\begin{equation}
\label{eq:mu1}
	\deig_1 (\mathcal{G}) = \inf \left\{ \frac{\int_\mathcal{G} |f'|^2\,\textrm{d}x}{\int_\mathcal{G} |f|^2\,\textrm{d}x}:f\in H^1_0(\mathcal{G})\right\};
\end{equation}
again, there is equality if and only if $f$ is a corresponding eigenfunction. As is standard, we shall call the quotient appearing in \eqref{eq:lambda1}--\eqref{eq:mu1} the \emph{Rayleigh quotient} (of the function $f$).

Finally, if $\psi$ is an eigenfunction associated with any one of the eigenvalues $\neig_k(\mathcal{G})$ or $\deig_k(\mathcal{G})$, $k\geq 1$, then we call the closures of the connected components of the set
\begin{displaymath}
	\{ x \in \mathcal{G}: \psi(x) \neq 0 \}
\end{displaymath}
the \emph{nodal domains} of the function $\psi$. (Any edges on which $\psi$ vanishes identically are considered not to lie in any nodal domain.) If $k\geq 2$, then $\psi$ must change sign in $\mathcal{G}$, since it is orthogonal in $L^2(\mathcal{G})$ to the eigenfunction associated with the smallest eigenvalue, which does not change sign and can easily be shown not to vanish anywhere (except on the set of Dirichlet vertices in the case of $\deig_1$).

We refer to both the monographs \cite{bk13,gs06,m14} as well as the introductions and preliminary sections of \cite{bkkm18,bkkm17,kkmm16} etc.\ for more details on these preliminaries.

We now collect the tools that we will need in the sequel. These are all based purely on the variational characterisations \eqref{eq:lambda1}, \eqref{eq:mu1} of the eigenvalues; most are from \cite{bkkm18} although some have appeared in various guises throughout the recent literature. We start with the most elementary, which we take from \cite[Theorem~3.4]{bkkm18} but which has also appeared elsewhere.

\begin{lemma}
\label{lem:join}
Suppose the graph ${\widetilde{\mathcal{G}}}$ is formed from $\mathcal{G}$ by gluing together two vertices $v_1,v_2\in \mathcal{V}(\mathcal{G})$, i.e., every edge that had $v_1$ or $v_2$ as an endpoint in $\mathcal{G}$ has a new common vertex $v_0 \in \mathcal{V}({\widetilde{\mathcal{G}}})$. Then $\deig_k ({\widetilde{\mathcal{G}}}) \geq \deig_k (\mathcal{G})$ and $\neig_k ({\widetilde{\mathcal{G}}}) \geq \neig_k (\mathcal{G})$ for all $k\geq 1$. For $\deig_1$ (corresp.~$\neig_2$) equality holds if and only if there is an eigenfunction $\psi$ of $\deig_1(\mathcal{G})$ (corresp.~$\neig_2 (\mathcal{G})$) such that $\psi(v_1)=\psi(v_2)$. In this case, the image of $\psi$ under the gluing procedure remains an eigenfunction of $\deig_1({\widetilde{\mathcal{G}}})$ (corresp.~$\neig_2 ({\widetilde{\mathcal{G}}})$).
\end{lemma}

\begin{proof}
The inequality follows from the identification of $H^1({\widetilde{\mathcal{G}}})$ as a subspace of $H^1(\mathcal{G})$, but such that the form \eqref{eq:form} itself is the same. The characterisation of equality follows from the fact that the minimum in \eqref{eq:lambda1} (corresp.~\eqref{eq:mu1}) is achieved if and only if the function is a corresponding eigenfunction. This is also a special case of \cite[Theorem~3.4]{bkkm18}; the inequality itself has appeared in multiple places including \cite[Theorem~3.1.8]{bk13}.
\end{proof}

\begin{lemma}
\label{lem:extend}
Suppose the graph ${\widetilde{\mathcal{G}}}$ is formed from $\mathcal{G}$ by lengthening an edge in $\mathcal{G}$. Then $\deig_1 ({\widetilde{\mathcal{G}}}) \leq \deig_1 (\mathcal{G})$ and $\neig_2 ({\widetilde{\mathcal{G}}}) \leq \neig_2 (\mathcal{G})$. Each of the inequalities is strict if there is a corresponding eigenfunction on $\mathcal{G}$ which does not vanish identically on the edge in question.
\end{lemma}

\begin{proof}
This is contained in \cite[Corollary~3.12(1)]{bkkm18}.
\end{proof}

\begin{lemma}
\label{lem:nodal}
Suppose $\psi$ is an eigenfunction corresponding to $\neig_k (\mathcal{G})$, $k\geq 2$, and suppose $\psi$ has $m\geq 2$ nodal domains, which we denote by $\mathcal{G}_1, \ldots, \mathcal{G}_m$. If we equip each of the $\mathcal{G}_i$ with Dirichlet conditions on the (finite) set $\mathcal{G}_i \cap \{x \in \mathcal{G}: \psi(x)=0\}$, then $\neig_k (\mathcal{G}) = \deig_1 (\mathcal{G}_i)$ and $\psi|_{\mathcal{G}_i}$ is an eigenfunction corresponding to $\deig_1 (\mathcal{G}_i)$, for all $i=1,\ldots,m$.
\end{lemma}

\begin{proof}
Fix $i=1,\ldots,m$. Since $\psi|_{\mathcal{G}_i} \in H^1_0 (\mathcal{G}_i)$ is a valid test function on $\mathcal{G}_i$, whose Rayleigh quotient is seen to be equal to its Rayleigh quotient on $\mathcal{G}$, i.e., $\neig_k(\mathcal{G})$. Thus $\neig_k (\mathcal{G}) \geq \deig_1 (\mathcal{G}_i)$. Conversely, since $\psi|_{\mathcal{G}_i}$ satisfies the strong form of the eigenvalue equation, including the zero condition, it must be an eigenfunction of $\deig_j (\mathcal{G}_i)$ for \emph{some} $j\geq 1$. But it does not change sign on $\mathcal{G}_i$; in fact, it is strictly different from zero everywhere on $\mathcal{G}_i$ outside the set of Dirichlet vertices. Now $\deig_1 (\mathcal{G}_i)$ is immediately seen to have an eigenfunction $\varphi_i$ not changing sign on $\mathcal{G}_i$ (if $\varphi$ is an eigenfunction, just replace it by $|\varphi|$ in \eqref{eq:mu1}). Hence the $L^2(\mathcal{G}_i)$-inner product of $\psi|_{\mathcal{G}_i}$ and $\varphi_i$ is not zero. Since both are eigenfunctions of the same Dirichlet eigenvalue problem, they must belong to the same eigenspace. It follows that $\psi|_{\mathcal{G}_i}$ corresponds to $\deig_1 (\mathcal{G}_i)$ and $\neig_k (\mathcal{G}) = \deig_1 (\mathcal{G}_i)$.
\end{proof}

We also have the following statement, which is complementary to Lemma~\ref{lem:nodal} and an immediate consequence of the min-max principle \eqref{eq:higher-ev}. It relates $\mu_k(\mathcal{G})$ to \emph{$k$-partitions} of $\mathcal{G}$. In the context of domains and manifolds, there is a large literature relating such partitions to the eigenvalues of the underlying domain or manifold and the nodal count of the associated eigenfunctions. We refer to \cite{bnh17} for a recent survey; for quantum graphs less has been done, but we refer to \cite{bbrs12,kklm19}.

\begin{lemma}
\label{lem:partition}
Suppose $\mathcal{H}_1,\mathcal{H}_2,\ldots,\mathcal{H}_k$ form a \emph{partition} of $\mathcal{G}$, that is, $\mathcal{H}_1,\mathcal{H}_2,\ldots,\mathcal{H}_k$ are closed graphs whose intersection is at most a finite set. Assume that $\partial \mathcal{H}_i = \{x \in \mathcal{G}: x \in \mathcal{H}_i \cap \overline{\mathcal{G} \setminus \mathcal{H}_i}\}$ is equipped with Dirichlet conditions, $i=1,2,\ldots,k$. Then $\neig_k (\mathcal{G}) \leq \max \{ \deig_1 (\mathcal{H}_1), \deig_1 (\mathcal{H}_2),\ldots, \deig_1 (\mathcal{H}_k)\}$.
\end{lemma}

\begin{proof}
Denote by $\varphi_i \in H^1_0 (\mathcal{H}_i)$ any eigenfunction corresponding to $\deig_1(\mathcal{H}_i)$, $i=1,\ldots,k$. Extend $\varphi_i$ by zero on the rest of $\mathcal{G}$ to obtain a function $\tilde\varphi_i$ in $H^1(\mathcal{G})$ whose Rayleigh quotient is still $\deig_1 (\mathcal{H}_i)$. Note that for $i \neq j$ the sets where $\tilde\varphi_i \neq 0$ and $\tilde\varphi_j \neq 0$ are disjoint in $\mathcal{G}$, so in particular the functions are linearly independent. The desired inequality now follows immediately upon taking $M := \aufspan \{\tilde\varphi_1,\ldots,\tilde\varphi_k \}$ in \eqref{eq:higher-ev}.
\end{proof}

We already observed that when $k=2$, any corresponding eigenfunction $\psi$ has (at least) two nodal domains. In the case $k \geq 3$, the precise number of nodal domains will be important to us. We thus recall the following general result from \cite[Theorem~2.6]{b08}.

\begin{lemma}
\label{lem:nodal-count}
Fix $k\geq 1$ and suppose $\mu_k (\mathcal{G})$ is simple and its eigenfunction $\psi$ does not vanish at any vertex $v \in \mathcal{V}$. Then the number $m$ of nodal domains of $\psi$ is bounded by
\begin{equation}
\label{eq:nodal-count}
	k - \beta \leq m \leq k,
\end{equation}
where we recall that $\beta = |\mathcal{E}|-|\mathcal{V}|+1$ is the (first) Betti number of $\mathcal{G}$.
\end{lemma}

We now give two surgery lemmata which will be central to the proof of Theorem~\ref{thm:dumbbell-diameter}. The first shows us that altering a graph by transferring ``mass'' from where its eigenfunction is smaller to where it is larger lowers the eigenvalue, and is adapted from \cite[Theorem~3.18(1)]{bkkm18}.

\begin{lemma}[Transplantation lemma]
\label{lem:transplantation}
Suppose $\psi \geq 0$ is an eigenfunction corresponding to $\deig_1(\mathcal{G})$. Suppose there is a vertex $v \in \mathcal{V} (\mathcal{G})$ and edges $e_1,\ldots,e_k \in \mathcal{E} (\mathcal{G})$ such that
\begin{equation}
\label{eq:transplantation-condition}
	\sup \{\psi(x):x \in e_1 \cup \ldots \cup e_k \} \leq \psi(v),
\end{equation}
and the total length of these edges is $|e_1|+\ldots+|e_k|=\ell>0$. Form a new graph ${\widetilde{\mathcal{G}}}$ from $\mathcal{G}$ by deleting the edges $e_1,\ldots,e_k$ (deleting also any vertices of degree one) and inserting new pendant edges at $v$ and/or lengthening existing edges in $\mathcal{G}$ to which $v$ is incident; any Dirichlet vertices in $\mathcal{G}$ not deleted should be preserved in ${\widetilde{\mathcal{G}}}$. Suppose that the total length of the additions and extensions is equal to or greater than $\ell$. Then $\deig_1 ({\widetilde{\mathcal{G}}}) \leq \deig_1 (\mathcal{G})$. The inequality is strict provided $\psi(v)>0$.
\end{lemma}

By \emph{deleting an edge}, we always mean removing the edge in question \emph{without} gluing its endpoints together; in particular, this process could disconnect the graph. See Figure~\ref{fig:transplantation}.

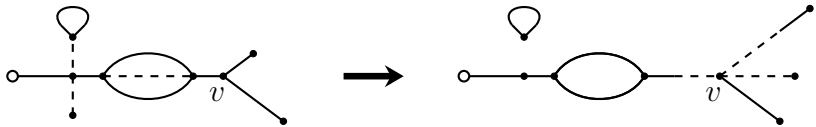
\begin{figure}[H]
\begin{tikzpicture}[scale=0.8]
\coordinate (a) at (-1,0);
\coordinate (b) at (0.5,0);
\coordinate (c) at (2,0);
\coordinate (d) at (2.5,0);
\draw[fill] (b) circle (1.5pt);
\draw[fill] (c) circle (1.5pt);
\draw[fill] (d) circle (1.5pt);
\draw[fill] (0,-0.65) circle (1.5pt);
\draw[fill] (0,0.65) circle (1.5pt);
\draw[fill] (0,0) circle (1.5pt);
\draw[thick] (a) -- (b);
\draw[thick] (c) -- (d);
\draw[thick,dashed] (0,-0.65) -- (0,0.65);
\draw[thick,dashed] (b) -- (c);
\draw[thick,bend left=60] (b) edge (c);
\draw[thick,bend right=60] (b) edge (c);
\draw[thick] (d) -- (3,0.375);
\draw[thick] (d) -- (3.5,-0.75);
\draw[fill] (3,0.375) circle (1.5pt);
\draw[fill] (3.5,-0.75) circle (1.5pt);
\draw[thick] (0,0.65) -- (-0.2,0.85);
\draw[thick] (0,0.65) -- (0.2,0.85);
%\draw[thick,bend left=30] (-0.2,0.85) edge (0.2,0.85);
\draw[thick] (-0.2,0.85) .. controls (-0.5,1.25) and (0.5,1.25) .. (0.2,0.85);
\draw[thick,fill=white] (a) circle (2.5pt);
\node at (2.4,0) [anchor=north] {$v$};
\draw[-{Stealth[scale=0.5,angle'=60]},line width=2.5pt] (4.5,0) -- (5.5,0);
\coordinate (d) at (6.5,0);
\coordinate (e) at (8,0);
\coordinate (f) at (9.5,0);
\coordinate (g) at (10.75,0);
\draw[fill] (e) circle (1.5pt);
\draw[fill] (f) circle (1.5pt);
\draw[fill] (g) circle (1.5pt);
\draw[thick] (d) -- (e);
\draw[thick] (f) -- (10,0);
\draw[thick,dashed] (10,0) -- (g);
\draw[thick,bend left=60] (e) edge (f);
\draw[thick,bend right=60] (e) edge (f);
\draw[thick,dashed] (g) -- (11.75,0.75);
\draw[thick] (11.75,0.75) -- (12.25,1.125);
\draw[thick,dashed] (g) -- (12,0);
\draw[thick] (g) -- (11.75,-0.75);
\draw[fill] (12.25,1.125) circle (1.5pt);
\draw[fill] (11.75,-0.75) circle (1.5pt);
\draw[fill] (12,0) circle (1.5pt);
\draw[thick,bend left=60] (e) edge (f);
\draw[thick,bend right=60] (e) edge (f);
\draw[thick] (7.5,0.65) -- (7.3,0.85);
\draw[thick] (7.5,0.65) -- (7.7,0.85);
%\draw[thick,bend left=30] (-0.2,0.85) edge (0.2,0.85);
\draw[thick] (7.3,0.85) .. controls (7,1.25) and (8,1.25) .. (7.7,0.85);
\draw[fill] (7.5,0.65) circle (1.5pt);
\draw[fill] (7.5,0) circle (1.5pt);
\draw[thick,fill=white] (d) circle (2.5pt);
\node at (10.65,0) [anchor = north] {$v$};
\end{tikzpicture}
\caption{The graph on the left is transformed into the graph on the right by transplantation to $v$. On the left, the dashed lines indicate the edges to be deleted, while on the right, they represent the insertion of new, and lengthening of existing, edges. These are chosen in such a way that the total length is preserved, or increased.}
\label{fig:transplantation}
\end{figure}

\begin{proof}
This is actually an easy special case of \cite[Theorem~3.18(1)]{bkkm18}, which is also valid for $\neig_2$, for more general transplantation procedures, and for more general vertex conditions. Here, this follows simply by constructing a test function
\begin{displaymath}
	\varphi(x):=\begin{cases}\psi(x) \qquad &\text{if } x \in \mathcal{G} \cap {\widetilde{\mathcal{G}}},\\
	\psi(v) \qquad &\text{if } x \in {\widetilde{\mathcal{G}}} \setminus \mathcal{G}.\end{cases}
\end{displaymath}
Then condition \eqref{eq:transplantation-condition} guarantees that $\|\varphi\|_{L^2({\widetilde{\mathcal{G}}})} \geq \|\psi\|_{L^2(\mathcal{G})}$; while since $\varphi$ is constant on ${\widetilde{\mathcal{G}}} \setminus \mathcal{G}$ and identical to $\psi$ elsewhere, we obviously have $\|\varphi'\|_{L^2({\widetilde{\mathcal{G}}})} \leq \|\psi\|_{L^2(\mathcal{G})}$. The inequality now follows from \eqref{eq:mu1}.

The strictness if $\psi(v)>0$ holds because in this case $\varphi$, being locally equal to a nonzero constant, cannot be an eigenfunction of ${\widetilde{\mathcal{G}}}$; hence it has strictly larger Rayleigh quotient than $\deig_1({\widetilde{\mathcal{G}}})$.
\end{proof}

We finish with a perturbation formula giving the rate of change of a simple eigenvalue with respect to a perturbation of the edge lengths; such a formula is often referred to as being of \emph{Hadamard type}, by way of analogy with the formulae for the derivative of an eigenvalue on a domain with respect to shape perturbations. The following formula has appeared in the literature multiple times, possibly beginning with \cite{f05a}.

\begin{lemma}[Hadamard-type formula]
\label{lem:hadamard}
Let $\lambda$ be a simple eigenvalue of the Laplacian (with either all natural or some natural and some Dirichlet vertices), with eigenfunction $\psi$ normalised to have $L^2$-norm $1$. Then the quantity
\begin{equation}
\label{eq:pruefer-amplitude}
	\mathscr{E}_e := \lambda\psi(x)^2+\psi'(x)^2,\qquad x \in e,
\end{equation}
is constant on each edge $e \in \mathcal{E}$. Moreover,
\begin{enumerate}
\item The derivative of $\lambda$ with respect to the edge length $|e|$ exists and equals
\begin{displaymath}
	\frac{d\lambda}{d|e|} = -\mathscr{E}_e.
\end{displaymath}
\item In particular, the rate of change of $\lambda$ with respect to lengthening $e_1$ and shortening $e_2$ by the same amount is strictly negative if and only if
\begin{displaymath}
	\mathscr{E}_{e_1} > \mathscr{E}_{e_2}.
\end{displaymath}
\end{enumerate}
\end{lemma}

The quantity \eqref{eq:pruefer-amplitude} is called the \emph{Pr\"ufer amplitude} (of the eigenfunction $\psi$ on the edge $e$).

\begin{proof}
The formula in (1) may be found in \cite{f05a}, \cite[Appendix~A]{cdv15} and \cite[Lemma~5.2]{bl17} (probably among others). Part (2) is an immediate application, and can at any rate be found, together with (1), in \cite[Section~3.2]{bkkm18}.
\end{proof}

\section{Properties of stars and dumbbells}
\label{sec:stars}

In addition to the stars $\mathcal{S}_n$ of Definition~\ref{def:star} we will need a second, related class of graphs, which also appeared in \cite{kkmm16}.

\begin{definition}
\label{def:star-dumbbell}
\begin{enumerate}
\item Fix suitable numbers $n \in \mathbb{N}$ and $\ell_0>0$, $\ell_1,\ell_2\geq 0$. A \emph{star dumbbell} will for us be a graph consisting of a finite edge (a \emph{handle}) $e_0$ of length $\ell_0$ connecting two distinct vertices $v_1$ and $v_2$, to each of which are attached $n$ \emph{pendant edges} each of length $\ell_1$ at $v_1$, and a further $n$ pendant edges of length $\ell_2$ at $v_2$. We will denote such a graph by $\mathcal{D} = \mathcal{D} (\ell_0,\ell_1,\ell_2,n)$. The set of $n$ pendant edges at $v_i$ of length $\ell_i$ each will be denoted by $\mathcal{P}_i$, $i=1,2$.
\item A \emph{symmetric star dumbbell} is a star dumbbell with the additional property that $\ell_1=\ell_2$.
\end{enumerate}
\end{definition}

See Figure~\ref{fig:star-dumbbell}. If $L>0$ and $D \in (0,L)$ are fixed, for $n\geq 2$ sufficiently large the symmetric star dumbbell of the form
\begin{equation}
\label{eq:star-explicit}
	\mathcal{D}_n = \mathcal{D}_n (L,D) := \mathcal{D} \left(\frac{nD-L}{n-1}, \frac{L-D}{2(n-1)}, \frac{L-D}{2(n-1)}, n\right)
\end{equation}
has total length $L$ and diameter $D$, and is seen to consist of two identical copies of $\mathcal{S}_n = \mathcal{S} (\frac{L}{2},\frac{D}{2},n)$ imagined as being glued together at their respective Dirichlet vertices.

\begin{figure}[H]
\begin{tikzpicture}[scale=0.8]
\coordinate (a) at (0,0);
\coordinate (b) at (4,0);
\draw[fill] (a) circle (1.2pt);
\draw[fill] (b) circle (1.2pt);
\draw[thick] (a) -- (b);
\draw[thick] (a) -- (0,-0.5);
\draw[thick] (a) -- (-0.5,0);
\draw[thick] (a) -- (0,0.5);
\draw[thick] (a) -- (0.35,0.35);
\draw[thick] (a) -- (-0.35,-0.35);
\draw[thick] (a) -- (0.35,-0.35);
\draw[thick] (a) -- (-0.35,0.35);
\draw[thick] (b) -- (5,0);
\draw[thick] (b) -- (4,1);
\draw[thick] (b) -- (4,-1);
\draw[thick] (b) -- (4.71,0.71);
\draw[thick] (b) -- (4.71,-0.71);
\draw[thick] (b) -- (3.29,0.71);
\draw[thick] (b) -- (3.29,-0.71);
\draw[fill] (0,-0.5) circle (1.2pt);
\draw[fill] (-0.5,0) circle (1.2pt);
\draw[fill] (0,0.5) circle (1.2pt);
\draw[fill] (0.35,0.35) circle (1.2pt);
\draw[fill] (0.35,-0.35) circle (1.2pt);
\draw[fill] (-0.35,0.35) circle (1.2pt);
\draw[fill] (-0.35,-0.35) circle (1.2pt);
\draw[fill] (5,0) circle (1.2pt);
\draw[fill] (4,1) circle (1.2pt);
\draw[fill] (4,-1) circle (1.2pt);
\draw[fill] (4.71,0.71) circle (1.2pt);
\draw[fill] (4.71,-0.71) circle (1.2pt);
\draw[fill] (3.29,0.71) circle (1.2pt);
\draw[fill] (3.29,-0.71) circle (1.2pt);
\coordinate (c) at (9,0);
\coordinate (d) at (13,0);
\draw[fill] (c) circle (1.2pt);
\draw[fill] (d) circle (1.2pt);
\draw[thick] (c) -- (d);
\draw[thick] (c) -- (8.6,0);
\draw[fill] (8.6,0) circle (1.2pt);
\draw[thick] (c) -- (9,0.4);
\draw[fill] (9,0.4) circle (1.2pt);
\draw[thick] (c) -- (9,-0.4);
\draw[fill] (9,-0.4) circle (1.2pt);
\draw[thick] (c) -- (9.2,0.346);
\draw[fill] (9.2,0.346) circle (1.2pt);
\draw[thick] (c) -- (9.2,-0.346);
\draw[fill] (9.2,-0.346) circle (1.2pt);
\draw[thick] (c) -- (8.8,0.346);
\draw[fill] (8.8,0.346) circle (1.2pt);
\draw[thick] (c) -- (8.8,-0.346);
\draw[fill] (8.8,-0.346) circle (1.2pt);
\draw[thick] (c) -- (9.346,0.2);
\draw[fill] (9.346,0.2) circle (1.2pt);
\draw[thick] (c) -- (9.346,-0.2);
\draw[fill] (9.346,-0.2) circle (1.2pt);
\draw[thick] (c) -- (8.654,0.2);
\draw[fill] (8.654,0.2) circle (1.2pt);
\draw[thick] (c) -- (8.654,-0.2);
\draw[fill] (8.654,-0.2) circle (1.2pt);
\draw[thick] (d) -- (13.4,0);
\draw[fill] (13.4,0) circle (1.2pt);
\draw[thick] (d) -- (13,0.4);
\draw[fill] (13,0.4) circle (1.2pt);
\draw[thick] (d) -- (13,-0.4);
\draw[fill] (13,-0.4) circle (1.2pt);
\draw[thick] (d) -- (13.2,0.346);
\draw[fill] (13.2,0.346) circle (1.2pt);
\draw[thick] (d) -- (13.2,-0.346);
\draw[fill] (13.2,-0.346) circle (1.2pt);
\draw[thick] (d) -- (12.8,0.346);
\draw[fill] (12.8,0.346) circle (1.2pt);
\draw[thick] (d) -- (12.8,-0.346);
\draw[fill] (12.8,-0.346) circle (1.2pt);
\draw[thick] (d) -- (13.346,0.2);
\draw[fill] (13.346,0.2) circle (1.2pt);
\draw[thick] (d) -- (13.346,-0.2);
\draw[fill] (13.346,-0.2) circle (1.2pt);
\draw[thick] (d) -- (12.654,0.2);
\draw[fill] (12.654,0.2) circle (1.2pt);
\draw[thick] (d) -- (12.654,-0.2);
\draw[fill] (12.654,-0.2) circle (1.2pt);
\end{tikzpicture}
\caption{A star dumbbell with $n=7$ (left); the symmetric star dumbbell $\mathcal{D}_{11}$ (right): the handle has length $(11D-L)/10$, while all $22$ short pendant edges have length $(L-D)/20$ each.}
\label{fig:star-dumbbell}
\end{figure}
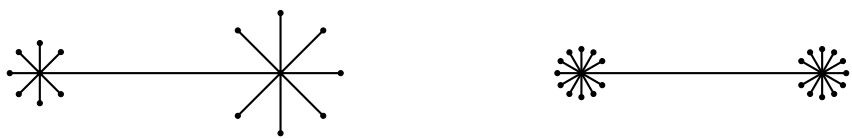
\vspace{-12pt}

The link between $\mathcal{S}_n$ and $\mathcal{D}_n$ is made more precise in Lemma~\ref{lem:dumbbell-and-star}; first, we need two lemmata describing the relevant eigenfunctions of $\mathcal{S}_n$ and $\mathcal{D}_n$.

\begin{lemma}
\label{lem:star-eigenfunction}
Let $\mathcal{S} = \mathcal{S} (L,D,n)$ be any star with $n\geq 2$ and $0<D\leq L$. The eigenfunction associated with $\deig_1 (\mathcal{S})$, when chosen positive, is strictly increasing away from the Dirichlet vertex $v_0$, with vanishing derivative only at the Kirchhoff vertices of degree one, and is invariant with respect to permutations of the $n$ identical edges of length $\ell_1$.
\end{lemma}

\begin{proof}
Note that $\deig_1 (\mathcal{S})$ is simple as it is the smallest eigenvalue and denote by $\psi$ the corresponding eigenfunction, chosen non-negative in $\mathcal{S}$. Let $e_1,\ldots, e_n$ be the $n$ equal edges of length $\ell_1$; then the function $\varphi \in H^1 (\mathcal{S})$ given by the average value
\begin{displaymath}
	\varphi = \frac{1}{n} \sum_{j=1}^n \psi|_{e_j}
\end{displaymath}
on each of $e_1,\ldots,e_n$ and $\varphi|_{e_0} = \psi|_{e_0}$ is invariant under permutations of the edges $e_1,\ldots,e_n$, and immediately seen to satisfy the (classical) eigenvalue equation for $\deig_1 (\mathcal{S})$. Since $\deig_1 (\mathcal{S})$ is simple, the only possibility is that $\varphi = \psi$ everywhere.

To show that $\psi$ is strictly monotone, it suffices to show that $\psi'$ cannot vanish identically at any interior point (at the central vertex $v_1$, this means ruling out $\psi|_{e_j}'(v_1)=0$ edgewise). Since $\psi$ cannot change sign on $\mathcal{S}$, nor be identically equal to a nonzero constant on a set of positive measure, if $\psi'$ vanishes, then $\psi$ has a strict interior maximum. In this case, it must reach a non-negative local minimum at either an interior point of one of the edges or a (Neumann--Kirchhoff) vertex $v$. In the first case, the (one-dimensional) maximum principle is violated directly. In the second case, since we certainly have $\psi|_{e_j}'(v)=0$ for all edges $e_j \sim v$, we may equally apply the one-dimensional maximum principle to obtain a contradiction.
\end{proof}

\begin{lemma}
\label{lem:dumbbell-eigenfunction}
Suppose $\mathcal{D} = \mathcal{D} (\ell_0,\ell_1,\ell_2,n)$ is any star dumbbell, $\ell_0,\ell_1,\ell_2>0$, $n\geq 1$. Assume that $\ell_0 > \max\{\ell_1,\ell_2\}$, i.e., the handle is longer than the pendant edges, \emph{or} that $\ell_1 = \ell_2$, i.e., $\mathcal{D}$ is symmetric. Then $\neig_2 (\mathcal{D})$ is simple and its eigenfunction $\psi$, unique up to scalar multiples, is invariant with respect to permutations of the edges within each pendant collection of edges $\mathcal{P}_i$, $i=1,2$.
\end{lemma}

\begin{proof}
Fix one of the $\mathcal{P}_i$. Then we may choose a basis of $L^2(\mathcal{D})$ made of eigenfunctions such that each either takes the value $0$ at the central vertex $v_i$ of $\mathcal{P}_i$ (the eigenfunction is ``odd''), or it is invariant with respect to permutations of the edges in $\mathcal{P}_i$ (it is ``even''). (Indeed, if $\psi$ is any eigenfunction and $e_1,\ldots,e_n$ are the edges of $\mathcal{P}_i$, then it suffices to consider instead the eigenfunction $(\psi|_{e_1}+\ldots+\psi|_{e_n})/n$ as in Lemma~\ref{lem:star-eigenfunction}, together with its orthogonal complement in the span of $\psi$.) We do this for both $\mathcal{P}_i$.

Equipped with this basis, we note that the eigenfunctions which are ``even'' with respect to both $\mathcal{P}_i$ are all simple within the space of all such ``even'' eigenfunctions, since their value at any point depends only on that point's position along any path of $\mathcal{D}$ realising the diameter (and thus they correspond to one-dimensional problems). Hence, to prove the lemma, it is sufficient to show that the smallest non-constant of these has a smaller eigenvalue than any of the ``odd'' eigenfunctions, each of the latter being supported without loss of generality only on one of the $\mathcal{P}_i$. Indeed, under the assumption $\ell_1 \geq \ell_2$, the smallest eigenvalue associated with an odd eigenfunction is $\pi^2/\ell_1^2$, corresponding to an eigenfunction each of whose two nodal domains corresponds to exactly one pendant edge of $\mathcal{S}_1$. If $\ell_0 > \ell_1$ or if $\ell_1 = \ell_2$, an elementary calculation shows that the (unique) zero of the even eigenfunction $\psi$ with the smallest eigenvalue must lie in the interior of the handle $e_0$. In particular, each edge of $\mathcal{P}_1$ is strictly contained in one the nodal domains of $\psi$; and thus the eigenvalue of $\psi$ is strictly smaller than $\pi^2/\ell_1^2$.
\end{proof}

\begin{lemma}
\label{lem:dumbbell-and-star}
Fix $0<D\leq L$ and $n\geq 2$. Denote by $\mathcal{D}_n (L,D)$ the symmetric star dumbbell with total length $L$ and diameter $D$. Then
\begin{displaymath}
	\neig_2 (\mathcal{D}_n (L,D)) = \deig_1 \left(\mathcal{S}\left(\frac{L}{2},\frac{D}{2},n\right)\right)
\end{displaymath}
and the two copies of $\mathcal{S}(\frac{L}{2},\frac{D}{2},n)$ embedded in $\mathcal{D}_n (L,D)$ are the two nodal domains of the eigenfunction corresponding to $\neig_2 (\mathcal{D}_n (L,D))$.
\end{lemma}

\begin{proof}
This follows immediately from Lemmata~\ref{lem:star-eigenfunction} and~\ref{lem:dumbbell-eigenfunction}, the latter in the form valid for symmetric star dumbbells. Alternatively, this statement is contained implicitly in \cite[Proof of Lemma~7.6]{kkmm16}.
\end{proof}

With this background, we can now give the proof of Proposition~\ref{prop:star}. As mentioned in the introduction, (1) and (2) were proved in \cite{kkmm16} (in a slightly different form), and to avoid repetition of the somewhat tedious calculations we will not give their proofs again here.

\begin{proof}[Proof of Proposition~\ref{prop:star}]
(1) and (2) For all $n \geq 2$, as shown in Lemma~\ref{lem:dumbbell-and-star}, we have
\begin{displaymath}
	\deig_1 (\mathcal{S}_n) = \neig_2 (\mathcal{D}_n)
\end{displaymath}
where $\mathcal{D}_n$ is the symmetric star graph having total length $2L$ and diameter $2D$. The statements (1) and (2) now follow from \cite[Lemma~7.6 and its proof]{kkmm16} and \cite[Remark~7.3(c)]{kkmm16}, respectively, bearing in mind that $\diam(\mathcal{D}_n)=2D$ and $|\mathcal{D}_n|=2L$. For an alternative proof of (1), see Lemma~\ref{lem:star-monotonicity}(1).

(3) We introduce the notation
\begin{displaymath}
	F(\omega,D) := \cos (\omega D) - \omega (L-D) \sin (\omega D).
\end{displaymath}
and wish to show that the derivative of $\omega$ with respect to $D$ is negative when $L>D$:
\begin{displaymath}
\begin{split}
	\frac{\partial\omega}{\partial D} = -\frac{\partial F}{\partial D}\Big/\frac{\partial F}{\partial\omega}
	&= \frac{\omega\sin (\omega D) - \omega \sin (\omega D) + \omega^2(L-D)\cos (\omega D)}
	{-D\sin (\omega D) - (L-D)\sin (\omega D) - \omega (L-D)D\cos (\omega D)}\\
	&= -\frac{\omega^2(L-D)\cos (\omega D)}{L\sin (\omega D) + \omega (L-D)D\cos (\omega D)}.
\end{split}
\end{displaymath}
Using the relation $\cos (\omega D) = \omega (L-D)\sin (\omega D)$, i.e., $F(\omega,D)\equiv 0$, which in particular implies that neither $\cos (\omega D)$ nor $\sin (\omega D)$ can be zero, we have
\begin{displaymath}
	\frac{\partial\omega}{\partial D} = - \frac{\omega^3(L-D)^2\sin(\omega D)}{(L+\omega^2(L-D)^2D)\sin(\omega D)}
	= - \frac{\omega^3(L-D)^2}{L+\omega^2(L-D)^2D} < 0
\end{displaymath}
since $\omega > 0$ by assumption.

(4) See \cite[Remark~7.3(a) and proof of Theorem~7.2]{kkmm16}.
\end{proof}

We next give two lemmata showing how the eigenvalues of stars and star dumbbells depend on changes in the parameters (total length, diameter etc.). The key tool in both is the Hadamard formula, Lemma~\ref{lem:hadamard}.

\begin{lemma}
\label{lem:star-monotonicity}
\begin{enumerate}
\item For fixed $L>0$ and $D \in (0,L)$, the function
\begin{displaymath}
	n \mapsto \deig_1 \left(\mathcal{S} \left(L, D, n \right)\right)
\end{displaymath}
is strictly decreasing in $n \geq 2$.
\item For fixed $L>0$ and $n\geq 2$, the function
\begin{displaymath}
	D \mapsto \deig_1 (\mathcal{S}(L,D,n))
\end{displaymath}
is strictly decreasing in $D \in (0,L)$.
\item Fix $L>0$, $D>0$ and $n\geq 2$. Then for each $L_1 > L$ there exists $n_1 \geq n$ such that
\begin{displaymath}
	\deig_1 (\mathcal{S}(L_1,D,n_1)) < \deig_1 (\mathcal{S}(L,D,n)).
\end{displaymath}
\end{enumerate}
\end{lemma}

\begin{proof}
(1) Although this could be derived from an analysis of the corresponding secular equation, cf.~the proof of Proposition~\ref{prop:star}, we will show how it can be obtained via the transplantation lemma~\ref{lem:transplantation}. Essentially the same proof will also yield (3).

Fix any numbers $n_2>n_1\geq 2$. Denote by $\psi$ the eigenfunction of $\deig_1(\mathcal{S}(L,D,n_1))$, chosen positive, by $v_1$ the central vertex (i.e., of degree $n_1+1$) of $\mathcal{S}(L,D,n_1)$, and by
\begin{displaymath}
	\ell_1 := \frac{L-D}{n_1-1} > \frac{L-D}{n_2-1} =: \ell_2
\end{displaymath}
the lengths of the identical edges of $\mathcal{S}(L,D,n_1)$ and $\mathcal{S}(L,D,n_2)$, respectively. Now by Lemma~\ref{lem:star-eigenfunction}, we know that $\psi$ takes on the same value at the $n_1$ points at distance $\ell_2$ to a degree one Kirchhoff vertex of $\mathcal{S}(L,D,n_1)$. Hence, by Lemma~\ref{lem:join}, if we glue these points together to create a new vertex $v_2$ of degree $2n_1$ (see Figure~\ref{fig:star-surgery}), $\deig_1 (\mathcal{S}(L,D,n_1))$ is unaffected and $\psi$ is still the eigenfunction; in particular, it is still a monotonically increasing function of the distance to the Dirichlet vertex.
\begin{figure}[H]
\begin{tikzpicture}[scale=0.8]
\coordinate (a) at (-1,0);
\coordinate (b) at (2,0);
\coordinate (c) at (3.5,0);
\draw[fill] (b) circle (2pt);
\draw[fill] (c) circle (1.5pt);
\draw[thick] (a) -- (c);
\draw[thick] (b) -- (3.061,1.061);
\draw[thick] (b) -- (3.061,-1.061);
\draw[fill] (3.061,1.061) circle (1.5pt);
\draw[fill] (3.061,-1.061) circle (1.5pt);
\draw[fill,red] (2.707,0.707) circle (2pt);
\draw[fill,red] (3,0) circle (2pt);
\draw[fill,red] (2.707,-0.707) circle (2pt);
\draw[thick,fill=white] (a) circle (2.5pt);
%\node at (0.5,0) [anchor=south] {$e_0$};
%\node at (-1,0) [anchor=east] {$v_0$};
\node at (2.1,0) [anchor=north east] {$v_1$};
\draw[-{Stealth[scale=0.5,angle'=60]},line width=2.5pt] (4.5,0) -- (5.5,0);
\coordinate (d) at (6.5,0);
\coordinate (e) at (9.5,0);
\coordinate (f) at (10.5,0);
\coordinate (g) at (11,0);
\draw[fill] (e) circle (2pt);
\draw[fill] (g) circle (1.5pt);
\draw[thick] (d) -- (g);
\draw[thick] (f) -- (10.854,0.354);
\draw[thick] (f) -- (10.854,-0.354);
\draw[fill] (10.854,0.354) circle (1.5pt);
\draw[fill] (10.854,-0.354) circle (1.5pt);
\draw[thick,bend left=60] (e) edge (f);
\draw[thick,bend right=60] (e) edge (f);
\draw[thick,fill=white] (d) circle (2.5pt);
\draw[fill, red] (f) circle (2pt);
\node at (9.6,0) [anchor=north east] {$v_1$};
\node at (10.5,-0.2) [anchor = north] {$v_2$};
\end{tikzpicture}
\caption{The star $\mathcal{S}(L,D,n_1)$ with the points to be glued together marked in red (left); the graph obtained after the gluing (right). To create $\mathcal{S}(L,D,n_2)$ out of the right-hand graph, we remove all but one of the edges between $v_1$ and $v_2$ and, in their place, create new pendant edges at the red vertex $v_2$.}
\label{fig:star-surgery}
\end{figure}
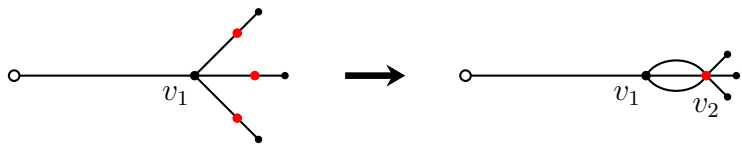

We now create $\mathcal{S}(L,D,n_2)$ out of this graph by deleting $n_1-1$ of the $n_1$ parallel edges of length $\ell_2-\ell_1$ each between $v_1$ and $v_2$ and, in their place, inserting $n_2-n_1$ pendant edges of length $\ell_2$ each at $v_2$. Since $\psi$ was smaller on the deleted parallel edges than at $v_2$, Lemma~\ref{lem:transplantation} yields
\begin{displaymath}
	\deig_1 (\mathcal{S}(L,D,n_2)) \leq \deig_1 (\mathcal{S}(L,D,n_1)).
\end{displaymath}
In fact, since $\psi(v_2)>0$ by Lemma~\ref{lem:star-eigenfunction}, this inequality is strict.

(2) We use the Hadamard formula, Lemma~\ref{lem:hadamard}. Write $\mathcal{S}$ for $\mathcal{S}(L,D,n)$, for given $L,D,n$. Noting that the simple eigenvalue $\deig_1$ is a differentiable function of the edge lengths at $\mathcal{S}$, if we write $\ell_0$ for the length of the Dirichlet edge $e_0$ and $\ell_1$ for the common length of each of the other $n$ edges, then we see that
\begin{equation}
\label{eq:star-diameter-hadamard}
	\frac{d}{dD} \deig_1 (\mathcal{S}) = \frac{d}{d\ell_0} \deig_1 (\mathcal{S}) - n \left(\frac{1}{n} \cdot \frac{d}{d\ell_1}\deig_1(\mathcal{S})\right)
\end{equation}
since increasing $D$ while holding $L$ and $n$ constant is equivalent to lengthening $e_0$ while shortening the $n$ other edges by $1/n$th of that amount each. As before, let $v_1$ denote the central vertex, $e_1,\ldots,e_n$ the $n$ identical edges and $\psi$ the eigenfunction. By Lemma~\ref{lem:hadamard}, we have
\begin{displaymath}
	\frac{d}{d\ell_0} \deig_1 (\mathcal{S}) = -\mathscr{E}_{e_0} = - \left(\deig_1(\mathcal{S}) \psi(v_1)^2 + \psi|_{e_0}'(v_1)^2 \right),
\end{displaymath}
where we recall $\psi|_{e_0}'(v_1)$ is the (normal) derivative of $\psi$ on $e_0$ at $v_1$; while by Lemma~\ref{lem:hadamard}, the continuity--Kirchhoff condition at $v_1$ and the symmetry property of $\psi$ from Lemma~\ref{lem:star-eigenfunction},
\begin{displaymath}
%\begin{aligned}
	\frac{d}{d\ell_1} \deig_1 (\mathcal{S}) = - \left(\deig_1(\mathcal{S}) \psi(v_1)^2 + \psi|_{e_j}'(v_1)^2 \right)
	= - \left(\deig_1(\mathcal{S}) \psi(v_1)^2 + \left(\frac{1}{n}\psi|_{e_0}'(v_1)\right)^2 \right)
%\end{aligned}
\end{displaymath}
(for any fixed $j=1,\ldots,n$). Inserting these expressions into \eqref{eq:star-diameter-hadamard}, we obtain
\begin{displaymath}
	\frac{d}{dD} \deig_1 (\mathcal{S}) = \left(\frac{1}{n^2}-1\right) \psi|_{e_0}'(v_1)^2.
\end{displaymath}
This last expression is strictly negative since $n\geq 2$ and $\psi'$ does not vanish at $v_1$ on any edge with which $v_1$ is incident by Lemma~\ref{lem:star-eigenfunction}.

(3) The proof is similar to the proof of (1), so we only sketch it. Choose any $n_1 \geq n$ such that the $n_1$ identical edges of $\mathcal{S}(L_1,D,n_1)$ are shorter than the $n$ identical edges of $\mathcal{S}(L,D,n)$, that is, such that
\begin{displaymath}
	\frac{L_1-D}{n_1-1} < \frac{L-D}{n-1}.
\end{displaymath}
As in (1), we may glue the $n$ vertices of $\mathcal{S}(L,D,n)$ at distance $(L_1-D)/(n_1-1)$ from a degree one Kirchhoff vertex to form $v_2$ without affecting the eigenvalue or eigenfunction. We may now transplant the surplus parallel edges from $v_1$ to $v_2$ to pendant edges at $v_2$ of the right length to create $\mathcal{S}(L_1,D,n_1)$ and strictly lower the eigenvalue in the process.
\end{proof}

\begin{lemma}
\label{lem:dumbbell-derivative}
For fixed total length $L$, fixed $\ell_0\geq \ell>0$ and fixed $n\geq 1$ consider the family of star dumbbells $\mathcal{D} = \mathcal{D} (\ell_0,\ell_1,\ell-\ell_1,n)$, where $\ell_1 \in [0,\ell]$. Then
\begin{enumerate}
\item $\frac{d}{d\ell_1}\neig_2(\mathcal{D})$ exists for all $\ell_1 \in (0,\ell)$ and is strictly negative if $\ell_1 \in (0,\ell/2)$ and strictly positive if $\ell_1 \in (\ell/2,\ell)$.
\item In particular, $\neig_2 (\mathcal{D})$ reaches its unique global minimum over $\ell_1 \in [0,\ell]$ at $\ell_1 = \ell/2$.
\end{enumerate}
\end{lemma}

In words, a symmetric star dumbbell has the lowest first eigenvalue among all star dumbbells having the same total length, diameter and number of pendant edges at each side. While the proof is similar to (parts of) the proof of Lemma~\ref{lem:star-monotonicity}, the latter lemma does not directly imply Lemma~\ref{lem:dumbbell-derivative} because, while the nodal domains of the eigenfunction of $\neig_2(\mathcal{D})$ will be stars, it would require more work to study their dependence on the edge lengths of $\neig_2(\mathcal{D})$; so instead we give a direct proof.

\begin{proof}
(1) The existence of the derivative follows immediately from Lemma~\ref{lem:hadamard}, which is applicable since $\neig_2(\mathcal{D})$ is always simple by Lemma~\ref{lem:dumbbell-eigenfunction}. By symmetry, it suffices to restrict attention to $\ell_1 \in (0,\ell/2)$ and prove that
\begin{equation}
\label{eq:star-lifting}
	\frac{d}{d\ell_1} \neig_2 (\mathcal{D}) < 0 \qquad \text{if } \ell \in (0,\ell_1/2).
\end{equation}
Denote by $\psi$ the corresponding eigenfunction, which is unique up to scalar multiples by Lemma~\ref{lem:dumbbell-eigenfunction}. Then  by Lemma~\ref{lem:hadamard} $\psi$ is identical on all edges within each of the stars, and to prove \eqref{eq:star-lifting} it suffices to prove that if $e_1$ is an edge in $\mathcal{P}_1$ and $e_2$ is an edge in $\mathcal{P}_2$, then $\ell_1 = |e_1| < |e_2| = \ell_2$ implies $\mathscr{E}_{e_1} > \mathscr{E}_{e_2}$. By Lemma~\ref{lem:hadamard}, this in turn is equivalent to showing
\begin{displaymath}
	\neig_2(\mathcal{D})\psi(v_1)^2 + \psi|_{e_1}' (v_1)^2 >
	\neig_2(\mathcal{D})\psi(v_2)^2 + \psi|_{e_2}' (v_2)^2,
\end{displaymath}
$i=1,2$. Denote by $e_0$ the handle, so that $e_1$ and $e_0$ are adjacent at $v_1$, and $e_2$ and $e_0$ are adjacent at $v_2$. Note that since $\ell_0\geq \ell >\max\{\ell_1, \ell-\ell_1\}$ the eigenfunction $\psi$ has exactly one zero, and this is on the handle $e_0$.

\emph{Claim:} The zero of $\psi$ is strictly closer to $v_2$ than $v_1$.

To prove the claim: if the claim does not hold, then, supposing the star $\mathcal{D}^+ := \{x \in \mathcal{D}: \psi(x) \geq 0\}$ to contain $\mathcal{P}_2$, and noting that $\neig_2 (\mathcal{D}) = \deig_1 (\mathcal{D}^+)$ by Lemma~\ref{lem:nodal}, we may reflect $\mathcal{D}^+$ across the set $\{x \in \mathcal{D}: \psi(x)=0\}$ to obtain a new (symmetric) star dumbbell $\widetilde{\mathcal{D}}$ such that, by symmetry, $\neig_2(\widetilde{\mathcal{D}}) = \deig_1 (\mathcal{D}^+)$. But the handle of $\widetilde{\mathcal{D}}$ is at least as long as $e_0$ and, since $\ell_1 < \ell/2$, the pendant edges of its other star are strictly longer than those of $\mathcal{D}$. But by Lemma~\ref{lem:extend}, this means that $\neig_2 (\widetilde{\mathcal{D}}) < \neig_2 (\mathcal{D})$, a contradiction. This proves the claim.

It follows from the claim that $\psi(v_1)^2>\psi(v_2)^2$; correspondingly, since, again, the Pr\"ufer amplitude is constant on each edge, we also have $\psi|_{e_0}'(v_1)^2 < \psi|_{e_0}'(v_2)^2$, so that
\begin{equation}
\label{eq:e0-to-add}
	-(n-1)\psi|_{e_0}'(v_1)^2 > -(n-1)\psi|_{e_0}'(v_2)^2.
\end{equation}
Now by the Kirchhoff condition and the fact that $\psi$ is identical on all pendant edges within each star,
\begin{multline*}
	\neig_2(\mathcal{D})\psi(v_1)^2 + n\psi|_{e_1}'(v_1)^2
	= \neig_2(\mathcal{D})\psi(v_1)^2 +\psi|_{e_0}'(v_1)^2\\
	= \neig_2(\mathcal{D})\psi(v_2)^2 +\psi|_{e_0}'(v_2)^2
	=\neig_2(\mathcal{D})\psi(v_2)^2 + n\psi|_{e_2}'(v_2)^2.
\end{multline*}
Adding \eqref{eq:e0-to-add} yields
\begin{displaymath}
	\neig_2(\mathcal{D})\psi(v_1)^2 + \psi|_{e_1}'(v_1)^2 > \neig_2(\mathcal{D})\psi(v_2)^2 + \psi|_{e_2}'(v_2)^2,
\end{displaymath}
as desired.

(2) This follows immediately from (1), also using the continuity of $\neig_2$ as $\ell_1 \to 0$ or $\ell$, for fixed $\ell_0$ and $n$.
\end{proof}

We finish this section with a kind of symmetrisation or balancing result for stars which we will need for the proof of Theorem~\ref{thm:diameter-higher}. This is closely related to the minimisation result of Lemma~\ref{lem:dumbbell-derivative}(2) when combined with Lemma~\ref{lem:dumbbell-and-star}.

\begin{lemma}
\label{lem:star-symmetrisation}
Suppose $\mathcal{S}_1$ and $\mathcal{S}_2$ are stars with diameter $D_1$, $D_2$ and total length $L_1\geq D_1$, $L_2\geq D_2$, respectively. Assume that both stars have $n$ identical shorter edges (each of length $(L_1-D_1)/n \geq 0$ and $(L_2-D_2)/n \geq 0$, respectively).\footnote{When we say \emph{shorter}, we are including the assumption that these edges are shorter than the respective $(n+1)$st edges equipped with the Dirichlet condition; that is, $\frac{L_i-D_i}{2(n-1)} \leq \frac{nD_i-L_i}{n-1}$, $i=1,2$.} Denote by $\mathcal{S}^\ast$ the star with diameter $(D_1 + D_2) / 2$, total length $(L_1 + L_2) / 2$, and $n$ identical shorter edges, and by $\mathcal{D}^\ast$ the symmetric star dumbbell with diameter $D_1+D_2$ and total length $L_1+L_2$, formed by gluing together two copies of $\mathcal{S}^\ast$ at their respective vertices. Then
\begin{equation}
\label{eq:star-symmetrisation}
	\max \{ \deig_1 (\mathcal{S}_1), \deig_1 (\mathcal{S}_2) \} \geq \neig_2 (\mathcal{D}^\ast) = \deig_1 (\mathcal{S}^\ast).
\end{equation}
The inequality is strict if $\mathcal{S}_1 \neq \mathcal{S}_2$.
\end{lemma}

\begin{proof}
We glue $\mathcal{S}_1$ and $\mathcal{S}_2$ together at their Dirichlet points to create a (non-symmetric) star dumbbell $\mathcal{D}$ having total length $L_1+L_2$ and diameter $D_1+D_2$. By Lemma~\ref{lem:partition} we have $\neig_2 (\mathcal{D}) \leq \max \{ \deig_1 (\mathcal{S}_1), \deig_1 (\mathcal{S}_2)$. Denote by $\mathcal{D}^\ast$ the symmetric star dumbbell having the same length and diameter as $\mathcal{D}$. Then $\neig_2 (\mathcal{D}^\ast) \leq \neig_2 (\mathcal{D})$ by Lemma~\ref{lem:dumbbell-derivative}(2). Appealing to Lemma~\ref{lem:dumbbell-and-star} completes the proof of \eqref{eq:star-symmetrisation}.

Now suppose that $\mathcal{S}_1 \neq \mathcal{S}_2$. We consider two cases: (1) the respective shorter edges have different lengths, i.e., $L_1-D_1 \neq L_2 - D_2$. In this case, $\mathcal{D} \neq \mathcal{D}^\ast$ and Lemma~\ref{lem:dumbbell-derivative}(2) yields the strict inequality $\neig_2 (\mathcal{D}^\ast) < \neig_2 (\mathcal{D})$; or (2) we have $L_1 - D_1 = L_2 - D_2$ so that $\mathcal{D} = \mathcal{D}^\ast$. In this case, since $\mathcal{S}_1 \neq \mathcal{S}_2$, we may assume without loss of generality that $L_1 < L_2$. In this case, $\mathcal{S}_1$ is strictly contained in the star $\mathcal{S}^\ast$ having length $(L_1+L_2)/2$ and diameter $(D_1+D_2)/2$, that is, $\mathcal{S}^\ast$ can be obtained from $\mathcal{S}_1$ by strictly lengthening the edge of the latter equipped with the Dirichlet vertex. The strictness statement in Lemma~\ref{lem:extend} now yields $\deig_1 (\mathcal{S}^\ast) < \deig_1 (\mathcal{S}_1)$.
\end{proof}

\section{Proof of the main theorems}
\label{sec:proof}

We now turn to the proof of Theorems~\ref{thm:dumbbell-diameter} and~\ref{thm:diameter-higher}. The key to both is the following observation.

\begin{lemma}
\label{lem:key}
Suppose $\mathcal{H}$ is a connected, compact graph with a finite number of edges, and with total length $L>0$ and equipped with a non-empty set of Dirichlet vertices $\mathcal{V}_{\mathcal{D}}$, and set
\begin{displaymath}
	d:= \sup_{x \in \mathcal{H}} \dist (x, \mathcal{V}_{\mathcal{D}}).
\end{displaymath}
Let $\mathcal{S}_n = \mathcal{S} (L,d,n)$ be the star having $n\geq 2$ identical edges, total length $L$ and diameter $d$. Then there exists $n_0 \geq 1$ such that
\begin{equation}
\label{eq:key}
	\deig_1 (\mathcal{S}_n) \leq \deig_1 (\mathcal{H})
\end{equation}
for all $n\geq n_0$. Equality in \eqref{eq:key} for some $n\geq 1$ implies that $L=d$ and $\mathcal{H}$ is a path graph (interval) of length $d$ with a Dirichlet condition at one endpoint and a Neumann condition at the other.
\end{lemma}

We explicitly remark that $\mathcal{H}$ is itself allowed to be a star graph of the type we are considering; in this case, the lemma contains a proof of the statement that $\deig_1 (\mathcal{S}_n) < \deig_1 (\mathcal{S}_m)$ if $n > m$ is sufficiently large (where $D$ and $L$ are fixed).

\begin{proof}
We may assume without loss of generality that $\mathcal{V}_{\mathcal{D}} = \{v_0\}$ consists of a single vertex of degree possibly larger than one, by formally gluing together all vertices in $\mathcal{V}_{\mathcal{D}}$ if necessary. Denote by $\psi$ the eigenfunction of $\deig_1 (\mathcal{H})$, chosen positive, and let $v$ be any point in $\mathcal{H}$ at which $\psi \in H^1 (\mathcal{H}) \hookrightarrow C(\mathcal{H})$ reaches a global maximum in $\mathcal{H}$; we assume without loss of generality that $v$ is a vertex.

By definition of $d$, there exists a path $\mathfrak{p}$ in $\mathcal{H}$ from $v$ to $v_0$ which has no self-intersections and length at most $d$. Assume that $\mathcal{H}$ is itself not a path (i.e., not an interval), so that $\mathfrak{p} \neq \mathcal{H}$ and $|\mathfrak{p}|<L$. Fix $n_0 \geq 1$, to be specified precisely later, but large enough that $n_0 > \deg v$ and the shortest edge in $\mathcal{H}$ is longer than
\begin{displaymath}
	\varepsilon := \frac{L - |\mathfrak{p}|}{n_0} > 0.
\end{displaymath}
We let $\widetilde{\mathcal{S}}$ be the star having one edge of length $|\mathfrak{p}|-\varepsilon$ (equipped with a Dirichlet condition at the far end) and $n$ shorter edges of length $\varepsilon$ each: then by construction, $|\widetilde{\mathcal{S}}| = L$ and $\diam (\widetilde{\mathcal{S}}) = |\mathfrak{p}| \leq d$.

We claim that $\deig_1 (\widetilde{\mathcal{S}}) < \lambda_1 (\mathcal{H})$; we will use the transplantation principle, Lemma~\ref{lem:transplantation}, to prove this. The lemma will then follow from Lemma~\ref{lem:star-monotonicity}: more precisely, part (2) yields the inequality $\deig_1 (\mathcal{S}_{n_0}) \leq \deig_1 (\widetilde{\mathcal{S}})$, while part (1) yields $\deig_1 (\mathcal{S}_n) \leq \deig_1 (\mathcal{S}_{n_0})$ for $n\geq n_0$.

To prove the claim, we look at the value
\begin{displaymath}
	m:= \max \{ \psi(x): x \in \mathcal{H} \text{ and } \dist (x,v) = \varepsilon \} > 0.
\end{displaymath}
We glue together all points $x \in \mathcal{H}$ such that $\psi(x)=m$ (of which there are only finitely many), to create a new vertex $v_\varepsilon$. In accordance with Lemma~\ref{lem:join}, this does not affect $\deig_1$ or $\psi$, so in a slight abuse of notation we will call the new graph $\mathcal{H}$.

Now the set $\{\psi\geq m\} \subset \mathcal{H}$ consists of a \emph{pumpkin} (collection of parallel edges) running from $v_\varepsilon$ to $v$, such that each edge of this pumpkin has length at most $\varepsilon$; and $v$ still lies on $\mathfrak{p}$ (or, more precisely, on its image under the gluing, which we will still denote by $\mathfrak{p}$). Note that the number of edges of this pumpkin is simply $\deg v < n_0$.

We now apply the transplantation lemma~\ref{lem:transplantation}. We remove every edge of $\mathcal{H}$ not on $\mathfrak{p}$ and not part of the pumpkin between $v_\varepsilon$ and $v$. In their place we first lengthen any edges between $v_\varepsilon$ and $v$ if necessary, so that each has exactly length $\varepsilon$. We then attach additional pendant edges each of length $\varepsilon$ to $v_\varepsilon$ until the new graph has total length $L$. (Note that the choice of $\varepsilon$ and the fact that the pumpkin has fewer than $n_0$ edges means that there is always enough material being transplanted to guarantee that all these edges can be made to have length exactly $\varepsilon$.)

We finally de-glue (cut through) $v$ to produce a graph having only pendant edges at $v_\varepsilon$. This graph is by construction $\widetilde{\mathcal{S}}$, and we have
\begin{equation}
\label{eq:after-the-operation}
	\deig_1 (\widetilde{\mathcal{S}}) \leq \deig_1 (\mathcal{H})
\end{equation}
by Lemmata~\ref{lem:join} (applied in reverse) and~\ref{lem:transplantation}. But under the assumption that $\mathcal{H}$ was not a path graph, the transplantation was nontrivial and so Lemma~\ref{lem:transplantation} in fact yields strict inequality in \eqref{eq:after-the-operation}. Combined with our earlier statements, this completes the proof.
\end{proof}

We can now give the proof of Theorem~\ref{thm:dumbbell-diameter}. In fact, in light of Proposition~\ref{prop:star}, more precisely, the fact that $\lambda_1 (\mathcal{S}_n)$ forms a decreasing sequence in $n$, which converges to $\omega^2$, it suffices to prove:

\begin{theorem}
\label{thm:alt-main}
Suppose $\mathcal{G}$ is any, connected compact graph with a finite number of edges, and with total length $L>0$ and diameter $D \in (0,L)$. Then there exists some $n\geq 1$ such that the star graph $\mathcal{S}_n$ having total length $L/2$ and diameter $D/2$ satisfies
\begin{displaymath}
	\deig_1 (\mathcal{S}_n) < \neig_2(\mathcal{G}).
\end{displaymath}
\end{theorem}

\begin{proof}[Proof of Theorem~\ref{thm:alt-main}, and hence of Theorem~\ref{thm:dumbbell-diameter}]
Let $\psi$ be any eigenfunction associated with $\mu_2 (\mathcal{G})$ and denote by $\mathcal{G}^+$ and $\mathcal{G}^-$ any two nodal domains of $\psi$. Then, by Lemma~\ref{lem:nodal},
\begin{displaymath}
	\neig_2 (\mathcal{G}) = \deig_1 (\mathcal{G}^+) = \deig_1 (\mathcal{G}^-)
\end{displaymath}
(where the Dirichlet vertices correspond to the points where $\psi = 0$), and with $\psi|_{\mathcal{G}^\pm}$ being the corresponding eigenfunctions. Moreover, $|\mathcal{G}^+| + |\mathcal{G}^-| \leq L$ and, if
\begin{equation}
\label{eq:dirichlet-diameter}
\begin{aligned}
	d^+ :\! &= \sup \{\dist (x,\mathcal{V}_{\mathcal{D}} (\mathcal{G}^+)) : x \in \mathcal{G}^+\}\\ &\equiv \sup \{ \dist (x,\{\psi=0\}): x \in \mathcal{G}^+ \}\\
	d^- :\! &= \sup \{\dist (x,\mathcal{V}_{\mathcal{D}} (\mathcal{G}^-)) : x \in \mathcal{G}^-\}
\end{aligned}
\end{equation}
(where it makes no difference whether we take the distance in $\mathcal{G}$ or in $\mathcal{G}^\pm$), then, since the distance from any point in $\mathcal{G}^+$ to any point in $\mathcal{G}^-$ is at most $D = \diam(\mathcal{G})$, we have
\begin{displaymath}
	d^+ + d^- \leq D.
\end{displaymath}
By Lemma~\ref{lem:key}, there exist stars $\mathcal{S}_n^+$ and $\mathcal{S}_n^-$ (for some $n$ sufficiently large, which is the same for both stars) with total lengths $|\mathcal{G}^+|$ and $|\mathcal{G}^-|$ and diameters $d^+$ and $d^-$, respectively, such that $\deig_1 (\mathcal{S}_n^\pm) \leq \deig_1 (\mathcal{G}^\pm)$. By Lemma~\ref{lem:star-monotonicity}(2) and (3), we may in fact assume without loss of generality that $|\mathcal{S}_n^+| + |\mathcal{S}_n^-| = L$ and $d^+ + d^- = D$ (possibly at the cost of making $n$ larger). Now, by Lemma~\ref{lem:star-symmetrisation} (or by a direct application of Lemmata~\ref{lem:partition} and~\ref{lem:dumbbell-derivative}(2) to the union of $\mathcal{S}_n^+$ and $\mathcal{S}_n^-$), we conclude that
\begin{equation}
\label{eq:alt-main-chain}
	\neig_2 (\mathcal{G}) \geq \max \{ \deig_1 (\mathcal{S}_n^+), \deig_1 (\mathcal{S}_n^-) \} \geq \deig_1 (\mathcal{S}_n),
\end{equation}
where $\mathcal{S}_n$ is now the star with length $L/2$ and diameter $D/2$.

It remains to prove that at least one inequality in \eqref{eq:alt-main-chain} is strict. Since $D<L$ by assumption, $\mathcal{G}$ is not a path. Suppose first that at least one of its nodal domains $\mathcal{G}^\pm$ is also not a path. If neither is a path with one Dirichlet and one Neumann endpoint, then Lemma~\ref{lem:key} already yields the strict inequality
\begin{displaymath}
	\neig_2 (\mathcal{G}) > \max \{ \deig_1 (\mathcal{S}_n^+), \deig_1 (\mathcal{S}_n^-) \}.
\end{displaymath}
If one is a path with one Dirichlet and one Neumann endpoint, say $\mathcal{G}^+$, then since the same is not true of $\mathcal{G}^-$, the star $\mathcal{S}_n^+$ is trivially equal to the path $\mathcal{G}^+$, while $\mathcal{S}_n^-$ is nontrivial (not a path). Since $\mathcal{S}_n^+ \neq \mathcal{S}_n^-$, Lemma~\ref{lem:star-symmetrisation} implies that the second inequality in \eqref{eq:alt-main-chain} is strict.

Finally, we deal with the case where $\mathcal{G}$ is not a path but it only has nodal domains which are paths: in this case, we must have $|\mathcal{G}^+|+|\mathcal{G}^-| < L$ and hence $|\mathcal{S}_n^+|+|\mathcal{S}_n^-|<L$. Assuming $\mathcal{S}_n$ still to have length $L/2$, strict inequality in Lemma~\ref{lem:star-monotonicity}(3) leads to strict inequality in the second inequality in \eqref{eq:alt-main-chain}.
\end{proof}

We conclude with the proof of Theorem~\ref{thm:diameter-higher}.

\begin{proof}[Proof of Theorem~\ref{thm:diameter-higher}]
Suppose first that $\neig_k (\mathcal{G})$ is simple and its eigenfunction $\psi$ does not vanish identically on any edge of $\mathcal{G}$. Then by Lemma~\ref{lem:nodal-count} $\psi$ has $m \geq k - \beta$ nodal domains $\mathcal{G}_1,\ldots,\mathcal{G}_m$, which by Lemma~\ref{lem:nodal} satisfy
\begin{displaymath}
	\neig_k (\mathcal{G}) = \deig_1 (\mathcal{G}_1) = \ldots = \deig_1 (\mathcal{G}_m)
\end{displaymath}
(with the Dirichlet vertices at the points where $\psi = 0$, and $\psi|_{\mathcal{G}_i}$ is, up to scalar multiples, the unique eigenfunction on $\mathcal{G}_i$, $i=1,\ldots,m$. Note that
\begin{displaymath}
	\sum_{i=1}^m |\mathcal{G}_i| = L
\end{displaymath}
since $\psi$ does not vanish identically on any edge. For each $i$, analogous to \eqref{eq:dirichlet-diameter}, set
\begin{displaymath}
	d_i := \sup \{ \dist (x,\mathcal{V}_{\mathcal{D}} (\mathcal{G}_j)) : x \in \mathcal{G}_i \};
\end{displaymath}
then, as in the case $k=2$, for each pair $i\neq j$, we have
\begin{equation}
\label{eq:higher-nodal-diameter-sum}
	d_i + d_j \leq D;
\end{equation}
Fix $n\geq 1$ sufficiently large. Then by Lemma~\ref{lem:key}, for each $i$ there exists a star $\mathcal{S}_n^i$ (as usual having $n$ identical shorter sides and one longer Dirichlet side) such that $|S_n^i| = |\mathcal{G}_i|$, $\diam (\mathcal{S}_n^i) = d_i$, and $\neig_k (\mathcal{G}) \geq \deig_1 (\mathcal{S}_n^i)$.

Now choose any pair $i_1 \neq j_1$ and apply Lemma~\ref{lem:star-symmetrisation} to $S_n^{i_1}$ and $S_n^{j_1}$, replacing them with the resulting stars which have the same total length and whose sum of diameters is the same, but which have smaller eigenvalues. Now choose a different pair $(i_2,j_2) \neq (i_1,j_1)$ and repeat.

Repeating this process arbitrarily often and passing to the limit, the stars converge (and their eigenvalues converge from above) to $m$ copies of the star $\mathcal{S}_n$ with total length $L/m$ and diameter no larger than $D/2$ (by \eqref{eq:higher-nodal-diameter-sum}); by Lemma~\ref{lem:star-monotonicity}(2) we may assume without loss of generality that actually $\diam (\mathcal{S}_n) = D/2$; and by Lemma~\ref{lem:star-symmetrisation}, we also have
\begin{displaymath}
	\neig_k (\mathcal{G}) \geq \max \{\deig_1 (\mathcal{S}_n^1),\ldots, \deig_1 (\mathcal{S}_n^m) \} \geq \deig_1 (\mathcal{S}_n).
\end{displaymath}
Since $m \geq k-\beta$ and, by Lemma~\ref{lem:star-monotonicity}(3), $\deig_1 (\mathcal{S}_n)$ is a decreasing function of increasing its length $L/m \mapsto L/(k-\beta)$ if its diameter $D/2$ is fixed (possibly at the cost of increasing $n$ at the same time), we obtain the statement of the theorem under the assumption that $\neig_k (\mathcal{G})$ is simple and $\psi$ does not vanish identically on any edge.

In the general case, we use a standard approximation argument. Let $\mathcal{G}$ be a connected, compact graph with a finite number of edges, such that $\mathcal{G}$ does not contain any loops longer than $D$. Firstly, if $\mathcal{G}$ does in fact contain any loops, we cut through the midpoint of each loop. Our assumption on the maximal loop length implies that this does not change either $L$ or $D$ and can only lower $\neig_k$ by Lemma~\ref{lem:join}. So we may assume without loss of generality that $\mathcal{G}$ does not contain any loops at all.

Now, by \cite[Theorem~3.6]{beli17}, there exists a sequence of graphs $\mathcal{G}_i$ having the same topology as $\mathcal{G}$, such that all edge lengths of $\mathcal{G}_i$ converge to those of $\mathcal{G}$, meaning in particular that $D_i := \diam (\mathcal{G}_i) \to D$ and $L_i := |\mathcal{G}_i| \to L$; and, for each $i$, we have that $\neig_k (\mathcal{G}_i)$ is simple and its eigenfunction does not vanish identically on any edge of $\mathcal{G}_i$. Now $\neig_k (\mathcal{G}_i)$ satisfies the eigenvalue bound of Theorem~\ref{thm:diameter-higher} for all $i$ (with $L_i$ and $D_i$ in place of $L$ and $D$); but, since this bound depends smoothly on $L$ and $D$, passing to the limit we obtain the desired bound for $\mathcal{G}$.
\end{proof}

\section{Concluding remarks}
\label{sec:remarks}

The idea of the proofs of Theorems~\ref{thm:dumbbell-diameter} and~\ref{thm:diameter-higher} consists in comparing each of the nodal domains of a graph $\mathcal{G}$ (more precisely, the nodal domains of a given eigenfunction $\psi$ associated with $\neig_k (\mathcal{G})$) with a corresponding star graph having the same total length and a possibly smaller diameter; this is the idea behind Lemma~\ref{lem:key}. To obtain the overall infimum, the balancing results of Section~\ref{sec:stars} show that the minimum over the $m$ stars obtained from the $m$ nodal domains of $\psi$ is achieved when the stars all have the same total length ($L/m$ each) and diameter ($D/2$ each). These copies can be pasted together at their respective Dirichlet vertices to form the graphs which, in the limit, converge to $m$-stars with point masses of size $L/m-D/2$ at each pendant vertex. The assumption that $L$ be sufficiently large compared with $D$ is, we believe, natural: it is necessary to ensure that these point masses actually have positive mass; in the borderline case where $L/m=D/2$, we obtain exactly the equilateral star which in accordance with \eqref{eq:friedlander} is minimising for $\neig_m (\mathcal{G})$ among all graphs $\mathcal{G}$ having given total length but without any constraint on the diameter.

The shortcoming in Theorem~\ref{thm:diameter-higher} is that in general we cannot expect that $m=k$, i.e., that the eigenfunction $\psi$ have $k$ nodal domains. Instead, we rely on the (sharp) lower bound $m\geq k-\beta$ in Lemma~\ref{lem:nodal-count}, which is also only valid ``generically'', that is, possibly after an arbitrarily small perturbation of the edge lengths, and for graphs without loops (in the presence of loops, there will always be special eigenfunctions supported on the loops, which cannot be eliminated by a perturbation argument).

In the case of trees, Lemma~\ref{lem:nodal-count} and hence Theorem~\ref{thm:diameter-higher} is sharp; all we lose via the edge perturbation argument is the ability to conclude that the inequality is always strict, that is, that there is no actual tree whose eigenvalue is equal to the square of the solution of \eqref{eq:higher-star-limit}. (However, we strongly expect this conclusion to be true.)

For non-trees, we do not expect Theorem~\ref{thm:diameter-higher} to be sharp. Indeed, simple examples such as loops and tadpoles suggest that Theorem~\ref{thm:diameter-higher} should be true in a sharper form, namely without the presence of $\beta$ and without the assumption that $\mathcal{G}$ not contain any long loops:

\begin{conjecture}
\label{conjecture}
Let $\mathcal{G}$ be \emph{any} connected, compact graph with total length $L$ and diameter $D$, where $L/k > D/2$. Then $\neig_k (\mathcal{G})$ is strictly larger than the square of the smallest positive solution $\omega>0$ of the equation
\begin{equation}
\label{eq:conjecture}
	\cos \left(\frac{\omega D}{2}\right) = \omega \left(\frac{L}{k}-\frac{D}{2}\right)\sin \left(\frac{\omega D}{2}\right).
\end{equation}
\end{conjecture}

(Here, again, we see the necessity of the assumption $L/k>D/2$ in \eqref{eq:conjecture} in order for this result to make sense.)

In other contexts, such as the proof of \eqref{eq:friedlander} or the related \cite[Theorem~4.7]{bkkm17}, one typically circumvents the problem of having too few nodal domains by first cutting through cycles in $\mathcal{G}$ to obtain a tree with the same total length, smaller eigenvalues (cf.~Lemma~\ref{lem:join}), and (generically) the correct number of nodal domains. Here, this is generally impossible since by cutting through a cycle one may increase the total diameter (see Figure~\ref{fig:cycle-problem} for an example). Actually, one only needs to guarantee the weaker property \eqref{eq:higher-nodal-diameter-sum} of the nodal domains of the cut graph, but there seems no reasonable way to arrange this.
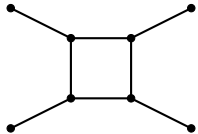
\begin{figure}[H]
\begin{tikzpicture}[scale=0.8]
\coordinate (a) at (-0.5,0.5);
\coordinate (b) at (0.5,0.5);
\coordinate (c) at (0.5,-0.5);
\coordinate (d) at (-0.5,-0.5);
\draw[thick] (a) -- (b);
\draw[thick] (b) -- (c);
\draw[thick] (c) -- (d);
\draw[thick] (d) -- (a);
\draw[fill] (a) circle (1.75pt);
\draw[fill] (b) circle (1.75pt);
\draw[fill] (c) circle (1.75pt);
\draw[fill] (d) circle (1.75pt);
\draw[thick] (a) -- (-1.5,1);
\draw[fill] (-1.5,1) circle (1.75pt);
\draw[thick] (b) -- (1.5,1);
\draw[fill] (1.5,1) circle (1.75pt);
\draw[thick] (c) -- (1.5,-1);
\draw[fill] (1.5,-1) circle (1.75pt);
\draw[thick] (d) -- (-1.5,-1);
\draw[fill] (-1.5,-1) circle (1.75pt);
\end{tikzpicture}
\caption{An example of a graph with a cycle, such that cutting the cycle at any point would increase the diameter.}
\label{fig:cycle-problem}
\end{figure}
We therefore leave Conjecture~\ref{conjecture} as an open problem; we also leave completely open the question of determining what happens when the assumption $L/k>D/2$ is not satisfied.

\bibliographystyle{plain}

\begin{thebibliography}{30}

\bibitem{ast17}
R.~Adami, E.~Serra and P.~Tilli,
\emph{Negative energy ground states for the $L^2$-critical NLSE on metric graphs},
Comm. Math. Phys. \textbf{352} (2017), 387--406.

\bibitem{ast15a}
R.\ Adami, E.\ Serra, and P.~Tilli,
\emph{Lack of ground state for {NLSE} on bridge-type graphs}, pp.~1--11 in
D.~Mugnolo (ed.), Mathematical Technology of Networks (Proc.\ Bielefeld 2013),
volume 128 of {\em Proc.\ Math.\ \& Stat.}, Springer, Cham, 2015.

\bibitem{ast15b}
R.~Adami, E.~Serra and P.~Tilli,
\emph{NLS ground states on graphs}, Calc. Var. \textbf{54} (2015), 743--761.

\bibitem{assw17}
M.~Aizenman, H.~Schanz, U.~Smilansky, and S.~Warzel,
\emph{Edge switching transformations of quantum graphs},
Acta Phys.\ Polon.\ A \textbf{132} (2017), 1699--1703.

\bibitem{ampr03}
W.~Arendt, G.~Metafune, D.~Pallara, and S.~Romanelli,
\emph{The {L}aplacian with {W}entzell--{R}obin boundary conditions on spaces of continuous functions},
Semigroup Forum \textbf{67} (2003), 247--261.

\bibitem{a16}
S.~Ariturk,
\emph{Eigenvalue estimates on quantum graphs},
preprint (2016), arXiv:1609.07471.

\bibitem{bbrs12}
R.~Band, G.~Berkolaiko, H.~Raz and U.~Smilansky,
\emph{The number of nodal domains on quantum graphs as a stability index of graph partitions},
Comm.\ Math.\ Phys.\ \textbf{311} (2012), 815--832.

\bibitem{bl17}
R.~Band and G.~L\'evy,
\emph{Quantum graphs which optimize the spectral gap},
Ann.\ Henri Poincar\'e \textbf{18} (2017), 3269--3323.

\bibitem{b08}
G.~Berkolaiko,
\emph{A lower bound for nodal count on discrete and metric graphs},
Comm.\ Math.\ Phys.\ \textbf{278} (2008), 803--819.

\bibitem{bkkm18}
G.~Berkolaiko, J.~B.~Kennedy, P.~Kurasov and D.~Mugnolo,
\emph{Surgery principles for the spectral analysis of quantum graphs},
Trans.\ Amer.\ Math.\ Soc.\ \textbf{372} (2019), 5153--5197.
 
\bibitem{bkkm17}
G.~Berkolaiko, J.~B.~Kennedy, P.~Kurasov and D.~Mugnolo,
\emph{Edge connectivity and the spectral gap of combinatorial and quantum graphs},
J.\ Phys.\ A: Math.\ Theor.\ \textbf{50} (2017), 365201.

\bibitem{bk13}
G.~Berkolaiko and P.~Kuchment,
\emph{Introduction to quantum graphs}.
Math.~Surveys and Monographs vol.~186,
American Mathematical Society, Providence, RI, 2013.

\bibitem{beli17}
G.~Berkolaiko and W.~Liu,
\emph{Simplicity of eigenvalues and non-vanishing of eigenfunctions of a quantum graph},
J.\ Math.\ Anal.\ Appl.\ \textbf{445} (2017), 803--818.

\bibitem{bnh17}
V.~Bonnaillie-No\"el and B.~Helffer,
\emph{Nodal and spectral minimal partitions -- the state of the art in 2016},
Chapter 10 in A.~Henrot (ed.), \emph{Shape optimization and spectral theory},
De Gruyter Open, Warsaw-Berlin, 2017.

\bibitem{cdv15}
Y.~Colin de Verdi\`ere,
\emph{Semi-classical measures on quantum graphs and the Gau{\ss} map of the determinant manifold},
Ann.\ Henri Poincar\'e \textbf{16} (2015), 347--364.

\bibitem{d18}
S.~Dovetta,
\emph{Existence of infinitely many stationary solutions of the $L^2$-subcritical and critical NLSE on compact metric graphs},
J.~Differential Equations \textbf{264} (2018), 4806--4821.

\bibitem{fk98}
S.~Fallat and S.~Kirkland,
\emph{Extremizing algebraic connectivity subject to graph theoretic constraints},
Electron.\ J.\ Linear Algebra \textbf{3} (1998), 48--74.

\bibitem{f05}
L.~Friedlander,
\emph{Extremal properties of eigenvalues for a metric graph},
Ann.\ Inst.\ Fourier (Grenoble) \textbf{55} (2005), 199--211.

\bibitem{f05a}
L.~Friedlander,
\emph{Genericity of simple eigenvalues for a metric graph},
Israel J.\ Math.\ \textbf{146} (2005), 149--156.

\bibitem{gs06}
S.~Gnutzmann and U.~Smilansky,
\emph{Quantum graphs: Applications to quantum chaos and universal spectral statistics},
Adv.\ Phys.\ \textbf{55} (2006), 527--625.

\bibitem{h17}
A.~Henrot (ed.),
\emph{Shape optimization and spectral theory},
De Gruyter Open, Warsaw-Berlin, 2017.

\bibitem{h03}
A.~Henrot,
\emph{Minimization problems for eigenvalues of the Laplacian},
J.\ Evol.\ Equ. \textbf{3} (2003), 443--461.

\bibitem{kklm19}
J.~B.~Kennedy, P.~Kurasov, C.~Lena and D.~Mugnolo,
\emph{A theory of spectral partitions of metric graphs},
in preparation (2019).

\bibitem{kkmm16}
J.~B.~Kennedy, P.~Kurasov, G.~Malenov\'a and D.~Mugnolo,
\emph{On the spectral gap of a quantum graph},
Ann.~Henri Poincar\'e \textbf{17} (2016), 2439--2473.

\bibitem{km16}
J.~B.~Kennedy and D.~Mugnolo,
\emph{The {C}heeger constant of a quantum graph},
Conference proceedings of the joint 87th annual meeting of the GAMM 
and Deutsche Mathematiker-Vereinigung,
PAMM \textbf{16} (2016), 875--876.

\bibitem{kn14}
P.~Kurasov and S.~Naboko,
\emph{Rayleigh estimates for differential operators on graphs},
J.\ Spectr.\ Theory \textbf{4} (2014), 211--219.

\bibitem{kmn13}
P.~Kurasov, G.~Malenov{\'a}, and S.~Naboko,
\emph{Spectral gap for quantum graphs and their edge connectivity},
J.\ Phys.\ A: Math.\ Theor.\ \textbf{46} (2013), 275309.

\bibitem{lss18}
D.~Lenz, M.~Schmidt and P.~Stollmann,
\emph{Topological Poincar\'e type inequalities and bounds on the infimum of the spectrum for graphs},
preprint (2018), arXiv:1801.09279.

\bibitem{m91}
B.~Mohar,
\emph{Eigenvalues, diameter, and mean distance in graphs},
Graphs Combin.\ \textbf{7} (1991), 53--64.

\bibitem{m14}
D.~Mugnolo,
\emph{Semigroup Methods for Evolution Equations on Networks},
Springer-Verlag, Berlin 2014.

\bibitem{mr07}
D.~Mugnolo and S.~Romanelli,
\emph{Dynamic and generalized {W}entzell node conditions for network equations},
Math.\ Meth.\ Appl.\ Sci.\ \textbf{30} (2007), 681--706.

\bibitem{n87}
S.~Nicaise,
\emph{Spectre des r{\'e}seaux topologiques finis},
Bull.\ Sci.\ Math.\ (2) \textbf{111} (1987), 401--413.

\bibitem{p67}
L.~E.~Payne,
\emph{Isoperimetric inequalities and their applications},
SIAM Rev.\ \textbf{9} (1967), 453--488.

\bibitem{pr16}
L.~M.~Del~Pezzo and J.~D.~Rossi,
\emph{The first eigenvalue of the {$p$}-{L}aplacian on quantum graphs},
Anal.\ Math.\ Phys.\ \textbf{6} (2016), 365--391.

\bibitem{r17}
J.~Rohleder,
\emph{Eigenvalue estimates for the Laplacian on a metric tree},
Proc.\ Amer.\ Math.\ Soc.\ \textbf{145} (2017), 2119--2129.

\bibitem{rs18}
J.~Rohleder and C.~Seifert,
\emph{Spectral monotonicity for Schr\"odinger operators on metric graphs},
preprint (2018), arXiv:1804.01827.

\end{thebibliography}

\end{document}